\documentclass[a4paper,12pt]{amsart}

\usepackage{amsmath, amsthm,amssymb}

\allowdisplaybreaks

\newtheorem{thm}{Theorem}[section]
\newtheorem{cor}[thm]{Corollary}
\newtheorem{lem}[thm]{Lemma}
\newtheorem{prop}[thm]{Proposition}

\theoremstyle{definition}

\theoremstyle{remark}
\newtheorem{rmk}[thm]{Remark}

\newtheorem{example}[thm]{Example}

\numberwithin{equation}{section}

\newcommand{\N}{\mathbb{N}}

\newcommand{\Z}{\mathbb{Z}}

\newcommand{\E}{\mathcal{E}}
\newcommand{\F}{\mathcal{F}}
\newcommand{\G}{\mathcal{G}}

\newcommand{\OO}{\mathcal{O}}

\newcommand{\NN}{\N}

\newcommand{\ZZ}{\Z}
\newcommand{\RR}{\mathbb{R}}
\newcommand{\TT}{\mathbb{T}}

\def\Gg{\mathcal{G}}
\def\Tt{\mathcal{T}}

\newcommand{\CC}{\mathbb{C}}
\newcommand{\KMS}{\text{KMS}}
\newcommand{\TC}{\mathcal{T}C}
\newcommand{\Aut}{\operatorname{Aut}}
\newcommand{\clsp}{\overline{\text{span}}}
\newcommand{\PF}{\kappa}

\def\clsp{\operatorname{\overline{span}}}
\def\MCE{\operatorname{MCE}}

\def\Ext{\operatorname{Ext}}
\def\Per{\operatorname{Per}}

\begin{document}

\author[an Huef]{Astrid an Huef}
\author[Kang]{Sooran Kang}
\author[Raeburn]{Iain Raeburn}
\address{Department of Mathematics and Statistics, University of Otago, PO Box 56, Dunedin 9054, New Zealand.}
\email{\{astrid, sooran, iraeburn\}@maths.otago.ac.nz}
\title[Spatial realisations of KMS states]{\boldmath Spatial realisations of KMS states\\ on the $C^*$-algebras of higher-rank graphs}

\thanks{This research has been supported by 
the Marsden Fund of the Royal Society of New Zealand. We thank Aidan Sims for helpful conversations.}

\begin{abstract}
Several authors have recently been studying the equilibrium or KMS states on the Toeplitz algebras of finite higher-rank graphs. For graphs of rank one (that is, for ordinary directed graphs), there is a natural dynamics obtained by lifting the gauge action of the circle  to an action of the real line. The algebras of higher-rank graphs carry a gauge action of a higher-dimensional torus, and there are many potential dynamics arising from different embeddings of the real line in the torus.  Previous results show that there is nonetheless a ``preferred dynamics'' for which the system exhibits a particularly satisfactory phase transition, and that the unique KMS state at the critical inverse temperature can then be implemented by intregrating vector states against a measure on the infinite path space of the graph. Here we obtain a similar description of the KMS state at the critical inverse temperature for other dynamics. Our spatial implementation is given by integrating against a measure on a space of paths which are infinite in some directions but finite in others. Our results are sharpest for the algebras of rank-two graphs.
\end{abstract}
\date{1 October 2014}
\maketitle

\section{Introduction}%\label{sec-intro}

There has recently been renewed interest in the KMS states of dynamical systems associated to directed graphs \cite{KW, aHLRS1, aHLRSseq, CL} and their higher-rank analogues \cite{Yang1, Yang2,aHLRS2, aHLRS3}. For systems based on the Toeplitz algebra of the graph, there is a simplex of KMS$_\beta$ states at each inverse temperature $\beta$ larger than a critical value $\beta_c$; under additional hypotheses on the graph, this simplex collapses to a single KMS state at inverse temperature $\beta_c$. This last state often factors through a state of the graph algebra of the graph, which is then the only  KMS state of the graph algebra.

Both the Toeplitz algebra and graph algebra of a directed graph $E$ carry a natural gauge action of the circle $\TT$ which lifts via $t\mapsto e^{it}$ to a natural dynamics, and the results in \cite{KW, aHLRS1, aHLRSseq} are about this dynamics for finite $E$ (more general dynamics have been studied in \cite{EL,CL,IK}, for example). The critical inverse temperature $\beta_c$ is given in terms of  the spectral radius $\rho(A)$ of the vertex matrix $A$ of the graph by $\beta_c=\ln\rho(A)$ (this goes back to \cite{EFW}). 

For a higher-rank graph $\Lambda$ of rank $k$, the gauge action is an action of the $k$-torus $\TT^k$, and to get a dynamics we have to choose an embedding of the real line $\RR$ in $\TT^k$. The graph $\Lambda$ has $k$ vertex matrices $\{A_j:1\leq j\leq k\}$, and if the embedding is given by $t\mapsto e^{itr}$ for some $r\in (0,\infty)^k$, the critical inverse temperature is $\beta_c=\max_j\{r_j^{-1}\ln\rho(A_j)\}$. For $\beta>\beta_c$, the dynamics on the Toeplitz algebra again admits a simplex of KMS$_\beta$ states \cite[Theorem~6.1]{aHLRS2}. At $\beta=\beta_c$ it matters what $r$ is. The best results in \cite{aHLRS2} and \cite{aHLRS3} concern a \emph{preferred dynamics} in which $r= \big(\ln\rho(A_1),\dots,\ln\rho(A_k)\big)$, and for which we have $\beta_c=1=r_j^{-1}\ln\rho(A_j)$ for all $j$. Under strong irreducibility hypotheses on the graph, there is then a unique KMS$_1$ state on the Toeplitz algebra and on the graph algebra \cite[Theorem~7.2]{aHLRS2}; for more general graphs, uniqueness requires aperiodicity of the graph \cite[Corollary~10.3]{aHLRS3}.
 
Graph and Toeplitz algebras have large commutative subalgebras $D$ generated by range projections, and we expect, both from previous studies \cite{EL,LR, LN} and from general results in \cite{N}, that KMS states should be given by integrating vector states against measures on the spectrum of $D$. For $\beta>\beta_c$, this is indeed the case: the KMS$_\beta$ states are constructed in \cite[\S6]{aHLRS2} as weighted sums of vector states in the Toeplitz representation on the finite-path space, and the weights give atomic measures with the required property (see Remark~\ref{atomicmeas}). For the preferred dynamics, where $\beta_c=1$, Proposition~10.2 of \cite{aHLRS3} describes a measure on the infinite-path space such that the KMS$_1$ state is an integral of vector states for the infinite-path representation, and indeed that result was needed in \cite{aHLRS3} to prove existence for periodic graphs.  So finding such measures seems an interesting and potentially useful enterprise. In this paper we construct suitable measures for other dynamics, in which $\beta_c=r_j^{-1}\ln\rho(A_j)$ for some but not all $j$.

Suppose that $\Lambda$ is a finite $k$-graph. The finite-path space is just $\Lambda$ itself in the discrete topology. The infinite-path space $\Lambda^\infty$ consists of functors $x$ from a model graph based on $\N^k$ into $\Lambda$  \cite[\S2]{KP}, and has a compact Hausdorff topology (because $\Lambda$ is finite). Both path spaces sit naturally in the spectrum of $D$, but they are not all of it by any means: there are many ways to go to infinity in $\N^k$, and  $\Lambda^\infty$ is the part of the boundary in which we have gone to infinity in every direction. We will focus on $K:=\{j:r_j^{-1}\ln\rho(A_j)=\beta_c\}$, and our measures will live on the part of the boundary where we have gone to infinity in the directions in $K$, and not in the others. We work with some concretely defined \emph{semi-infinite path spaces}, instead of working explicitly inside the spectrum of $D$. 

\medskip
We begin with a section on preliminary material. We briefly review facts about KMS states and results from Perron-Frobenius theory that we later rely on. We then set out our conventions for higher-rank graphs and their vertex matrices, and  discuss the Toeplitz-Cuntz-Krieger algebra and $C^*$-algebra of a higher-rank graph. At the end of \S\ref{TCK-defn}, we discuss the dynamics $\alpha^r$ which we will be using throughout the paper.  In \S\ref{semiinf}  we investigate the full path space $W_\Lambda$ of a higher-rank graph $\Lambda$, building on the recent work of Webster \cite{W}. In particular, we discuss the semi-infinite path spaces, and realisations of certain subsets as inverse limits which we will use to build measures. We then discuss the semi-infinite path representations that we use in our spatial realisations of KMS states.

We begin our analysis of KMS states in \S\ref{sec-kms-states} by looking at KMS$_\beta$ states on Toeplitz algebras above the critical inverse temperature. The main analysis of these states remains that of \cite[Theorem~6.1]{aHLRS2}, but we make some minor improvements to the general results. Then in Remarks~\ref{atomicmeas} and \ref{nolimit} we motivate our later analysis by describing a spatial realisation for $\beta>\beta_c$, and examining why it breaks down at $\beta=\beta_c$. Our main results are formulated in Theorem~\ref{main_thm}, and most of \S\ref{sec-spatial} is devoted to its proof. In \S\ref{secnogaugeexpectation}, we consider another spatial construction of KMS states which works when the set $K=\{j:r_j^{-1}\ln\rho(A_j)=\beta_c\}$ is a singleton, and in particular for any non-preferred dynamics on the Toeplitz algebra of a $2$-graph. This is itself of some interest, since many of the most interesting examples of higher-rank graph algebras are those of $2$-graphs \cite{DY, PRRS, PRW}.

For the preferred dynamics, the KMS states on the Toeplitz algebra at critical inverse temperatures factor through the quotient map onto the graph algebra. Our KMS states factor through the quotient which imposes the Cuntz-Krieger relations for degrees in $\N^K$ (Proposition~\ref{KMS_combine}). This quotient looks rather like the relative graph algebras of Muhly and Tomforde \cite{MT}, but not at first sight like the relative higher-rank graph algebras of Sims \cite{S1}. In Appendix~\ref{App:RCK}, we confirm that it is one of Sims' relative algebras. In our final Appendix~\ref{App:nesh}, we reconcile our results with Neshveyev's general machine for computing KMS states on groupoid algebras \cite{N}. Unfortunately, to do this we need an appropriate groupoid model for the Toeplitz algebras, and this does not seem to be explicitly available in the literature. So we provide one here, by adapting results of Yeend \cite{Y}, and then show that our measure is the quasi-invariant measure predicted by Neshveyev's theorem.

\section{Background material}\label{sec-background}

\subsection{KMS states}\label{secKMS}
Suppose that $(A,\alpha)$ is a dynamical system consisting of an action $\alpha$ of $\RR$ on a $C^*$-algebra $A$. As in \cite{BR,Ped}, we say that $a\in A$ is \emph{analytic} for $\alpha$ if the function $t\mapsto\alpha_t(a)$ extends to an analytic function $z\mapsto \alpha_z(a)$ on $\CC$ (and then that extension is automatically unique).
A state $\phi$ on $A$ is a \emph{KMS state with inverse temperature} $\beta$ (or a \emph{$\KMS_\beta$} state) of $(A,\alpha)$ if
\[
\phi(ab)=\phi(b\alpha_{i\beta}(a))
\]
for all analytic elements $a,b$. Proposition~8.12.3 of \cite{Ped} implies that it suffices to check the $\KMS$ condition on a set of analytic elements which span a dense subspace of $A$. In this paper, we are only interested in KMS$_\beta$ states with inverse temperature $\beta\in (0,\infty)$.

The following simple lemma will be handy when we want to normalise our dynamics. It says that changing the unit of time does not affect the behaviour of the system in any material way.

\begin{lem}\label{scaledynamics}
Suppose that $\alpha:\RR\to\Aut A$ is an action of $\RR$ on a $C^*$-algebra $A$ and that $\phi$ is a KMS$_\beta$ state of $(A, \alpha)$. Let $d\in (0,\infty)$ and define $\alpha':\RR\to\Aut A$ by $\alpha'_t=\alpha_{td}$. Then $\phi$ is a KMS$_{d^{-1}\beta}$ state of $(A, \alpha')$. 
\end{lem}

\subsection{The Perron-Frobenius theorem}
Let $X$ be a finite set. A matrix $A\in M_X(\CC)$ is \emph{irreducible} if for all $v,w\in X$ there exists $n\in\NN$ such that $A^n(v,w)\ne 0$. We say that a matrix is positive (non-negative) if all its entries are positive (non-negative). 

Let $A\in M_X([0,\infty))$ be an irreducible non-negative matrix. The Perron-Frobenius theorem says that the spectral radius $\rho(A)$ of  $A$ is an
eigenvalue of $A$ with a $1$-dimensional eigenspace and a  positive eigenvector (see, for example,
\cite[Theoren~1.5]{Sen}). We call
the unique  positive eigenvector with eigenvalue $\rho(A)$ and unit $1$-norm  the
\emph{unimodular Perron-Frobenius eigenvector of $A$}.
A vector $\epsilon\in [0,\infty)^X$ is \emph{subinvariant} for $A$ and $t \in \RR$ if $A\epsilon\le t \epsilon$.
The subinvariance theorem \cite[Theorem~1.6]{Sen} says that if a vector $\epsilon\in [0,\infty)^X$ is subinvariant for $A$ and a positive real number $t$, then all the entries of $\epsilon$ are positive and $t\ge \rho(A)$; moreover, $t=\rho(A)$ if and only if $A\epsilon= t\epsilon$. 

\subsection{Higher-rank graphs and their vertex matrices}
Let $k\in\N$ with $k\geq 1$. We write $e_1,\dots,e_k$ for the generators of $\NN^k$ and $n_i$ for the $i^{th}$ coordinate of $n\in (\NN \cup \{\infty\})^k$. For $m,n\in (\NN \cup \{\infty\})^k$ we write $m\leq n$ if and only if $m_i\leq n_i$ for $1\leq i\leq k$.

Let $\Lambda$ be a $k$-graph with vertex set $\Lambda^0$ and degree functor $d:\Lambda\to \NN^k$, as in \cite{KP}.  
 We say that $\Lambda$ is \emph{finite} if $\Lambda^n:=d^{-1}(n)$ is finite for all $n\in\NN^k$. Except in the appendices, we consider only finite $k$-graphs in this paper.
For $v, w\in \Lambda^0$ and $n\in\NN^k$, we write, for example,  \[v\Lambda:=\{\lambda\in\Lambda: r(\lambda)=v\}\text{\ \ and\ \ }v\Lambda^n w :=\{\lambda\in\Lambda^n : r(\lambda)=v, s(\lambda)=w\}.\]
We say that $\Lambda$ has no sources if $v\Lambda^n\ne \emptyset$ for every $v\in\Lambda^0$ and $n\in\NN^k$.

\begin{example}\label{Omega}
Let $\Omega_k^0:=\NN^k$, $\Omega_k:=\{(p,q)\in \NN^k\times\NN^k : p\leq q\}$, define $r,s:\Omega_k\to \Omega_k^0$ by $r(p,q):=p$ and $s(p,q):=q$, define composition by $(p,q)(q,r)=(p,r)$, and define $d:\Omega_k\to\NN^k$ by $d(p,q):=q-p$. Then $\Omega_k$ is a $k$-graph with no sources.

For $n \in (\NN \cup \{\infty\})^k$ we denote by $\Omega_{k,n}$ the subgraph $\{(p,q) \in \NN^k\times \NN^k : q \le n\}$ of $\Omega_k$. 
\end{example}

For $1\le i\le k$, let  $A_i$ be the the matrix in $M_{\Lambda^0}(\NN)$ with entries $A_i(v,w)=|v\Lambda^{e_i}w|$; we call the $A_i$ the \emph{vertex matrices} of $\Lambda$. Since $(A_iA_j)(v,w)=|v\Lambda^{e_i+e_j}w|$, the factorisation property of $\Lambda$ implies that $A_iA_j=A_jA_i$, and we define
\[
A^n:=\prod_{i=1}^k A_i^{n_i}\quad\text{for $n\in\NN^k$. }
\] 
We say that $\Lambda$ is \emph{coordinatewise irreducible} if vertex matrix $A_i$ is irreducible for $1\le i\le k$. If $\Lambda$ is coordinatewise irreducible, then \cite[Lemma~2.1]{aHLRS2} implies that the unimodular Perron-Frobenius eigenvectors of the $A_i$  are all equal, and we call it the \emph{common Perron-Frobenius eigenvector of $\Lambda$}. We write $\rho(\Lambda)$ for the vector $\big(\rho(A_1),\dots,\rho(A_k)\big)$.

We visualise $k$-graphs as coloured graphs, by choosing $k$ different colours $c_1,\dots,c_k$, and viewing paths in $\Lambda^{e_i}$ as edges of colour $c_i$. (See \cite[Chapter~10]{R} for a discussion of how this relates to the factorisation property, and \cite[\S3]{HRSW} for details of the relationship between a $k$-graph and its underlying coloured graph.) When $k=2$, we view edges in $\Lambda^{e_1}$ as blue, and edges in $\Lambda^{e_2}$ as red.

\subsection{\boldmath The Toeplitz-Cuntz-Krieger $C^*$-algebra of a $k$-graph}\label{TCK-defn} Let $\Lambda$ be a finite $k$-graph with no sources. For $\mu,\nu\in \Lambda$, we write $\Lambda^{\text{min}}(\mu,\nu)$ for the set of $(\eta,\zeta)$ in $\Lambda\times\Lambda$ such that $\mu\eta=\nu\zeta$ and $d(\mu\eta)=d(\mu)\vee d(\nu)$.
 As in \cite{aHLRS2, RS1}, a \emph{Toeplitz-Cuntz-Krieger $\Lambda$-family} consists of partial isometries $\{T_\lambda:\lambda\in\Lambda\}$ such that
\begin{itemize}
\item[(T1)] $\{T_v:v\in\Lambda^0\}$ are mutually orthogonal projections;
\item[(T2)] $T_\lambda T_\mu=T_{\lambda\mu}$ whenever $s(\lambda)=r(\mu)$;
\item[(T3)] $T_\lambda^* T_\lambda=T_{s(\lambda)}$ for all $\lambda\in\Lambda$;
\item[(T4)] for all $v\in\Lambda^0$ and $n\in\NN^k$, we have
\[T_v\ge\sum_{\lambda\in v\Lambda^n}T_\lambda T_\lambda^*;\]
\item[(T5)] for all $\mu,\nu\in\Lambda$, we have
\[T_\mu^* T_\nu=\sum_{(\eta,\zeta)\in\Lambda^{\text{min}}(\mu,\nu)}T_\eta T_\zeta^*,\]
where by convention the sum over the empty set is $0$.
\end{itemize}
A Toeplitz-Cuntz-Krieger $\Lambda$-family is a \emph{Cuntz-Krieger $\Lambda$-family} if in addition we have
\begin{itemize}
\item[(CK)] $T_v=\sum_{\lambda\in v\Lambda^n}T_\lambda T_\lambda^*$ for all $v\in\Lambda^0$ and $n\in\NN^k$.
\end{itemize}

The Toeplitz algebra $\mathcal{T}C^*(\Lambda)$ of a $k$-graph $\Lambda$ is generated by a universal Toeplitz-Cuntz-Krieger $\Lambda$-family $\{t_\lambda:\lambda\in\Lambda\}$, and the standard arguments show that
\[
\Tt C^*(\Lambda)=\clsp\big\{t_\mu t_\nu^*:\mu,\nu\in \Lambda\big\}.
\]
The Cuntz-Krieger algebra $C^*(\Lambda)$ is the quotient of $\mathcal{T}C^*(\Lambda)$ by the ideal generated by
\[
\Big\{ t_v-\sum_{\lambda\in v\Lambda^n}t_\lambda t_\lambda^*: v\in\Lambda, n\in\NN^k\Big\}.
\]
The universal property gives a gauge action $\gamma$ of $\TT^k$ on $\mathcal{T}C^*(\Lambda)$ such that $\gamma_z(t_\lambda)=z^{d(\lambda)}t_\lambda$ (using multi-index notation, so that $z^n=\prod_{i=1}^kz_i^{n_i}$ for $z=(z_1,\dots,z_k)\in \TT^k$ and $n\in \Z^k$).

\section{Semi-infinite path spaces}\label{semiinf}

Let $\Lambda$ be a finite $k$-graph, let  $n\in\NN^k$ and consider the graph $\Omega_{k,n}$ of Example~\ref{Omega}.  Then each  $\lambda\in \Lambda^n$ gives a functor $x_\lambda:\Omega_{k,n}\to\Lambda$, as follows.
Take $p\leq q\leq n$, use the factorisation property to see that there are unique paths $\lambda'\in\Lambda^p$, $\lambda''\in \Lambda^{q-p}$ and $\lambda'''\in \Lambda^{n-q}$ such that $\lambda=\lambda'\lambda''\lambda'''$, and then define $x_\lambda(p,q):=\lambda(p,q):=\lambda''$. The map $\lambda\mapsto x_\lambda$ is a bijection from $\Lambda^n:=d^{-1}(n)$ onto the set of degree-preserving functors from $\Omega_{k,n}$ to $\Lambda$. We use this bijection to identify the two sets, and this identification motivates the definitions of infinite and semi-infinite paths.

Now let  $n \in (\NN \cup \{\infty\})^k$. 
Then we denote by $\Lambda^n$  the set of $k$-graph morphisms from $\Omega_{k,n}$ to $\Lambda$. (When $n\in \NN^k$ we had already identified the set of $k$-graph morphisms from $\Omega_{k,n}$ to $\Lambda$ with $\Lambda^n:=d^{-1}(n)$ in the paragraph above.) We write $d(x) = n$ whenever $x \in \Lambda^n$.  As usual, we write $\Lambda^\infty$ for the infinite-path space $\Lambda^{\infty,\dots,\infty}$ and call its elements \emph{infinite paths}.

We consider the \emph{path space} 
\[W_\Lambda:=\bigcup_{n\in (\NN\cup \{\infty\})^k}\Lambda^n.
\] 
For each $\lambda\in\Lambda$ and finite subset $G$ of $s(\lambda)\Lambda$ we write 
\begin{equation}\label{defn-cylindersets} Z(\lambda):=\{x\in W_\Lambda: x(0,d(\lambda))=\lambda\}\quad\text{and}\quad
Z(\lambda\setminus G):=Z(\lambda)\setminus \big(\textstyle{\bigcup_{\alpha\in G}}Z(\lambda\alpha)\big).
\end{equation}
Theorems 3.1~and~3.2 of \cite{W} show that the
$Z(\lambda \setminus G)$ form a basis for a locally compact Hausdorff
topology on $W_\Lambda$ (see also \cite[\S2]{PW} and \cite[\S3]{FMY}). Webster shows in the proof of \cite[Theorem~3.2]{W} that  $Z(v)$ is compact for $v\in\Lambda^0$.  Since $\Lambda$ is finite, $W_\Lambda=\bigcup_{v\in\Lambda^0} Z(v)$ is  compact.  Then we also have:

\begin{lem}\label{oversight} Let $\lambda\in\Lambda$ and  $G$ be a finite subset of $s(\lambda)\Lambda$.
Then $Z(\lambda\setminus G)$ is compact in $W_\Lambda$.
\end{lem}

\begin{proof}
Since $Z(r(\lambda))$ is compact, it suffices to show $Z(\lambda)$ and $Z(\lambda\setminus G)$  are closed.  Let $\{x_n\}\subset Z(\lambda)$  and $x_n\to x$.  Then $x_n=\lambda y_n$ for $y_n\in Z(s(\lambda))$.  Since $Z(s(\lambda))$ is compact, there is a convergent subsequence $y_{n_i}\to y\in Z(s(\lambda))$.  Now it is easy to see that $x_{n_i}=\lambda y_{n_i}\to \lambda y$. Since $W_\Lambda$ is Hausdorff, $x=\lambda y$ and $Z(\lambda)$ is closed.  

Similarly, let $\{x_n\}\subset Z(\lambda\setminus G)$  and $x_n\to x$. Then $x_n\in Z(\lambda)$. Since $W_\Lambda$ is compact, $Z(\lambda)$ is closed, and  then $x\in Z(\lambda)$. Suppose, by way of contradiction, that $x\in Z(\lambda\alpha)$ for some $\alpha\in G$. Since $Z(\lambda\alpha)$ is open, we have $x_n\in Z(\lambda\alpha)$ eventually, a contradiction. Hence $Z(\lambda\setminus G)$ is closed.
\end{proof}

We consider a nonempty subset $K$ of $\{1,\dots,k\}$, and set $J:=\{1,\dots, k\}\setminus K$. We view $\NN^k$ as $\NN^J\times \NN^K$, and for $n\in(\NN\cup\{\infty\})^k$, we write $n=(n_J,n_K)$ where $n_J\in(\NN\cup\{\infty\})^J$ and $n_K\in(\NN\cup\{\infty\})^K$.
For $m\in \NN^J$, we define
\begin{gather*}
\Lambda^{m,\infty_K}:=\big\{x \in W_\Lambda : d(x)_J=m\;\;\text{and}\;\; d(x)_K=\infty_K\big\}\quad\text{and}\quad
\partial^K\!\Lambda:=\textstyle{\bigcup_{m\in\NN^J}}\Lambda^{m,\infty_K}.
\end{gather*}  
We call elements of   $\partial^K\!\Lambda$ \emph{semi-infinite} paths.

For $m\in \N^J$ and $n\leq p\in \N^K$, we define $r_{n,p}:\Lambda^{m,p}\to \Lambda^{m,n}$ by $r_{n,p}(\lambda)=\lambda(0,(m,n))$. Then the factorisation property implies that for $n\leq p\leq q\in \N^K$ we have $r_{n,p}\circ r_{p,q}=r_{n,q}$. So when we view the finite sets $\Lambda^{m,n}$ as topological spaces with the discrete topology, $\{\Lambda^{m,n},r_{n,p}\}$ is an inverse system of compact Hausdorff spaces. Then the inverse limit 
\[
\textstyle{\varprojlim_{n\in \N^K}}\Lambda^{m,n}
\] is a compact Hausdorff space, which we can realise concretely as the subspace of the product $\prod_{n\in \N^K}\Lambda^{m,n}$ consisting of the elements $\{\lambda^n\}_{n\in \N^K}$ satisfying $r_{n,p}(\lambda^p)=\lambda^n$ for $n\leq p$.

\begin{prop}\label{idinvlim} 
 Let $m\in \NN^J$.
For each $\{\lambda^n\}\in \varprojlim_{n\in \N^K}\Lambda^{m,n}$ there is a path $x^\lambda\in \Lambda^{m,\infty_K}$ such that $x^\lambda(0,(m,n))=\lambda^n$ for all $n\in \N^K$, and then $\phi:\{\lambda^n\}\mapsto x^\lambda$ is a homeomorphism of $\varprojlim_{n\in \N^K}\Lambda^{m,n}$ onto the subset $\Lambda^{m,\infty_K}$ of $W_{\Lambda}$.  In particular, $\Lambda^{m,\infty_K}$ is compact Hausdorff.
\end{prop}

For the proof, we need a lemma.

\begin{lem}\label{lem:top'}
The cylinder sets $\{Z(\lambda)\cap \Lambda^{m,\infty_K}:\lambda\in\Lambda, d(\lambda)_J=m\}$ are a basis for the relative topology on $\Lambda^{m,\infty_K}\subset W_{\Lambda}$.
\end{lem}

\begin{proof}
Suppose that $x\in \Lambda^{m,\infty_K}$ and $Z(\tau\setminus G)$ is a basic open neighbourhood of $x$ in $W_\Lambda$ --- in other words,  $\tau\in \Lambda$, $G$ is a finite subset of $s(\tau)\Lambda$, and
\[
x\in Z(\tau\setminus G)=Z(\tau)\setminus \big(\textstyle{\bigcup_{\alpha\in G}}Z(\tau\alpha)\big).
\]
We have to find $\lambda\in \Lambda$ such that $d(\lambda)_J=m$, $x\in Z(\lambda)$ and
\begin{equation}\label{needsigma}
Z(\lambda)\cap \Lambda^{m,\infty_K}\subset  Z(\tau)\setminus \big(\textstyle{\bigcup_{\alpha\in G}}Z(\tau\alpha)\big).
\end{equation}
We note first that $x\in Z(\tau)$ implies that $d(\tau)_J\leq m$. Next, we observe that 
\[
\Lambda^{m,\infty_K}\cap Z(\tau\alpha)\not=\emptyset \Longrightarrow d(\tau\alpha)_J\leq m.
\]
Now we take
\[
n=\big(\textstyle{\bigvee_{\alpha\in G,\;d(\tau\alpha)_J\leq m}}d(\tau\alpha)_K\big)\vee d(\tau)_K.
\]
Then since $x\notin Z(\tau\alpha)$ for all $\alpha\in G$, we have 
\[
x\in Z(x(0,(m,n)))\cap \Lambda^{m,\infty_K}\subset Z(\tau)\setminus \big(\textstyle{\bigcup_{\alpha\in G}}Z(\tau\alpha)\big),
\]
as required.
\end{proof}

\begin{proof}[Proof of Proposition~\ref{idinvlim}]
The existence of $x^\lambda$ is established as in \cite[Remarks~2.2]{KP} (which covers the case $J=\emptyset$). The factorisation property implies that for every $x\in \Lambda^{m,\infty_K}$, $\{\lambda^n\}=\{x(0,(m,n))\}$ belongs to $\varprojlim_{n\in \N^K}\Lambda^{m,n}$, and that $x\mapsto\{x(0,(m,n))\}$ is a set-theoretic inverse for $\phi$.

Let $\sigma\in \Lambda$ such that $d(\sigma)_J=m$. Then $\phi^{-1}(Z(\sigma))$ is the intersection of $\varprojlim_{n\in \N^K}\Lambda^{m,n}$ with the product set
\[
\Big\{\{\mu^n\}\in \prod_{n\in \N^K}\Lambda^{m,n}:\mu^n=\sigma(0,(m,n))\text{ for $n\leq d(\sigma)_K$}\Big\},
\]
and is therefore open in $\varprojlim_{n\in \N^K}\Lambda^{m,n}$. Since  $\{Z(\sigma)\cap \Lambda^{m,\infty_K}:d(\sigma)_J=m\}$ is a basis for the topology on $\Lambda^{m,\infty_K}$ by Lemma~\ref{lem:top'},  it follows that $\phi$ is continuous. Since the inverse limit is compact and $W_\Lambda$ is Hausdorff \cite[Theorem~3.2]{W}, $\phi$ is a homeomorphism onto its range $\Lambda^{m,\infty_K}$.
Since $\varprojlim_{n\in \N^K}\Lambda^{m,n}$ is  compact so is $\Lambda^{m,\infty_K}$.
\end{proof}

Next we build a Toeplitz-Cuntz-Krieger $\Lambda$-family $\{T_\lambda:\lambda\in\Lambda\}$ on $\ell^2(\partial^K\!\Lambda)$, and then the universal property of $\mathcal{T}C^*(\Lambda)$ gives a representation $\pi_T:\mathcal{T}C^*(\Lambda)\to B(\ell^2(\partial^K\!\Lambda))$.
Since $\pi_T$ depends on the partition $\{1,\dots,k\}=J\sqcup K$, we denote it by $\pi^K$, and we call it the \emph{semi-infinite path representation} for $K$.
We prove that $\pi^K$  factors through the quotient of $\mathcal{T}C^*(\Lambda)$ by the ideal $I^K$ generated by
\[
\Big\{t_v-\sum_{e\in v\Lambda^{e_i}}t_e t^*_e :v\in\Lambda^0,i\in K\Big\}.
\]

\begin{rmk}
The quotient $\mathcal{T}C^*(\Lambda)/I^K$ is an interesting example of the relative Cuntz-Krieger algebras $C^*(\Lambda;\E)$ of \cite[\S3]{S1}, and $\pi^K$ is one of the boundary-path representations in that paper. Since this last observation can only be directed to those familiar with \cite{S1}, we have put the details in Appendix~\ref{App:RCK}.
\end{rmk}

\begin{prop}\label{TCK}
Let $\Lambda$ be a finite $k$-graph with no sources. Let $J\sqcup K$ be a nontrivial partition of $\{1,\dots, k\}$.
 Let $\{\xi_{m,x} :m\in\NN^J,x\in\Lambda^{m,\infty_K}\}$ be the usual orthonormal basis of point masses in \[
\ell^2(\partial^K\!\Lambda)=\bigoplus_{m\in\NN^J}\ell^2(\Lambda^{m,\infty_K}).
\]
 For $\lambda\in \Lambda$, let $T_\lambda$ be the operator on $\ell^2(\partial^K\!\Lambda)$ such that
\[
T_\lambda \xi_{m,x}=\begin{cases}\xi_{m+d(\lambda)_J,\lambda x}\quad\text{if}\; s(\lambda)=r(x)\\
0\quad\quad\quad\quad\;\;\;\text{otherwise.}\end{cases}
\]
Then $\{T_\lambda:\lambda\in\Lambda\}$ is a Toeplitz-Cuntz-Krieger $\Lambda$-family  such that
\begin{equation*}%\label{eq:TCK}
\sum_{e\in v\Lambda^{e_i}}T_e T^*_e=T_v\quad \text{for $v\in\Lambda^0$ and $i\in K$.}
\end{equation*}
\end{prop}

\begin{proof}
Let $\lambda\in \Lambda$. Then the adjoint $T^*_\lambda$ is characterised by
\[
T^*_\lambda\xi_{m,x}=\begin{cases}\xi_{m-d(\lambda)_J, x(d(\lambda),\infty)}\quad \text{if}\;\; x(0,d(\lambda))=\lambda\;\;\text{and}\;\;m\ge d(\lambda)_J\\ 0\quad\quad\quad\quad\quad\quad\quad\;\; \text{otherwise.}\end{cases}
\]
Now it is easy to see that $T_\lambda T_\lambda^* T_\lambda=T_\lambda$, so $T_\lambda$ is a partial isometry. Let $v,w\in \Lambda^0$. Then $T_v$ is the projection onto $\clsp\{\xi_{m,x}:x\in\partial^K\!\Lambda, r(x)=v\}$, and $T_vT_w=0$ unless $v=w$. Thus $\{T_v:v\in\Lambda^0\}$ is a set of mutually orthogonal projections, and we have proved (T1).

For (T2), fix $\lambda,\mu\in\Lambda$ with $s(\lambda)=r(\mu)$. Then
\[\begin{split}
T_\lambda T_\mu\xi_{m,x}&=\begin{cases}T_\lambda \xi_{m+d(\mu)_J,\mu x}\quad\text{if}\;\;s(\mu)=r(x)\\
0\quad\quad\quad\quad\quad\quad\;\text{otherwise}\end{cases}\\
&=\begin{cases} \xi_{m+d(\mu)_J+d(\lambda)_J,\lambda\mu x}\quad\text{if}\;\;s(\mu)=r(x)\\
0\quad\quad\quad\quad\quad\quad\quad\;\;\;\text{otherwise}\end{cases}\\
&=T_{\lambda\mu}\xi_{m,x},\\
\end{split}\]
since $d(\lambda)_J +d(\mu)_J=d(\lambda\mu)_J$. To see (T3), take $\lambda\in\Lambda$. Then
\[\begin{split}
T^*_\lambda T_\lambda \xi_{m,x}&=\begin{cases}T^*_\lambda \xi_{m+d(\lambda)_J,\lambda x}\quad\text{if}\;\;s(\lambda)=r(x)\\
0\quad\quad\quad\quad\quad\quad\;\text{otherwise}\end{cases}\\
&=\begin{cases}\xi_{m,x}\quad\text{if}\;\;s(\lambda)=r(x)\\
0\quad\quad\;\text{otherwise}\end{cases}\\
&=T_{s(\lambda)}\xi_{m,x}.
\end{split}\]

We want to use (T5) to prove (T4), so we next establish (T5). Let $\mu, \nu\in\Lambda$ and $\xi_{m,x}, \xi_{n,y}\in \ell^2(\partial^K\!\Lambda)$. We will show that
\begin{equation}\label{eq-t5}
(T^*_\mu T_\nu \xi_{m,x}\mid \xi_{n,y})=\Big(\sum_{(\eta,\zeta)\in\Lambda^{\text{min}}(\mu,\nu)}T_\eta T^*_\zeta \xi_{m,x}\;\big|\; \xi_{n,y}\Big).
 \end{equation}
We consider three cases. First, suppose that $\Lambda^{\text{min}}(\mu,\nu)=\emptyset$. Then $\mu x\ne \nu y$ and 
\[
(T^*_\mu T_\nu \xi_{m,x}\mid \xi_{n,y})=(T_\nu \xi_{m,x}\mid T_\mu \xi_{n,y})=(\xi_{m+d(\nu)_J,\nu x}\mid \xi_{n+d(\mu)_J, \mu y})=0.
\]
The empty sum on the right-hand-side of  \eqref{eq-t5} is by definition zero,  and hence  \eqref{eq-t5} holds in this first case.

Second, suppose that $\Lambda^{\text{min}}(\mu,\nu)\ne\emptyset$ and  $(T^*_\mu T_\nu\xi_{m,x}\mid \xi_{n,y})\ne 0$.
Then
\[
0\neq (T^*_\mu T_\nu\xi_{m,x}\mid \xi_{n,y})=(T_\nu\xi_{m,x}\mid T_\mu \xi_{n,y})=(\xi_{m+d(\nu)_J,\nu x}\mid \xi_{n+d(\mu)_J,\mu y}).
\]
Thus $\xi_{m+d(\nu)_J,\nu x}=\xi_{n+d(\mu)_J,\mu y}$ and $(T^*_\mu T_\nu\xi_{m,x}\mid \xi_{n,y})=1$.
To see that
\[
\Big(\sum_{(\eta,\zeta)\in\Lambda^{\text{min}}(\mu,\nu)}T_\eta T^*_\zeta \xi_{m,x}\mid \xi_{n,y}\Big)=1,
\]
 we will find $(\sigma,\tau)\in\Lambda^{\text{min}}(\mu,\nu)$ such that $(T_\sigma T^*_\tau \xi_{m,x}\mid \xi_{n,y})=1$; this suffices because  then  $T^*_\zeta\xi_{m,x}=0$ and $T_\eta^* \xi_{n,y}=0$ for every other $(\eta,\zeta)\in \Lambda^{\text{min}}(\mu,\nu)$. Since $\xi_{m+d(\nu)_J,\nu x}=\xi_{n+d(\mu)_J,\mu y}$, we have
\[
m+d(\nu)_J=n+d(\mu)_J,\;\; \nu x=\mu y,\;\; r(x)=s(\nu)\;\; \text{\ and\ } \;\;r(y)=s(\mu).
\]
Let $\sigma=y(0,d(\mu)\vee d(\nu)-d(\mu))$ and $\tau=x(0,d(\mu)\vee d(\nu)-d(\nu))$.
Then \[\mu\sigma=(\mu y)(0, d(\mu)\vee d(\nu))=(\nu x)(0, d(\mu)\vee d(\nu))=\nu\tau.\] So $(\sigma, \tau)\in \Lambda^{\text{min}}(\mu,\nu)$.
Then $x=\tau x'$ and $y=\sigma y'$ for some $x', y'$. But $\nu\tau x'=\nu x=\mu y=\mu \sigma y'$ and $\nu\tau=\mu\sigma$, so $x'=y'$.
Also $m+d(\nu)_J=n+d(\mu)_J$ and $\nu\tau=\mu\sigma$ imply that $m-d(\tau)_J=n-d(\sigma)_J$. So we have
\[\begin{split}
(T_\sigma T^*_\tau \xi_{m,x}\mid \xi_{n,y})&=(T^*_\tau \xi_{m,x}\mid  T^*_\sigma\xi_{n,y})\\
&=(\xi_{m-d(\tau)_J,x(d(\tau),\infty)}\mid  \xi_{n-d(\sigma)_J, y(d(\sigma),\infty)})\\
&= (\xi_{m-d(\tau)_J, x'}\mid \xi_{n-d(\sigma)_J,y'})=1,
 \end{split}\]
as required.

Third, suppose that $\Lambda^{\text{min}}(\mu,\nu)\ne\emptyset$ and  $(T^*_\mu T_\nu\xi_{m,x}\mid \xi_{n,y})=0$. We argue by contradiction. Suppose that the right-hand-side of \eqref{eq-t5} is not zero. Then there exists $ (\sigma,\tau)\in\Lambda^{\text{min}}(\mu,\nu)$ such that $(T_\sigma T^*_\tau \xi_{m,x}\mid \xi_{n,y})\neq 0$.
 This implies that 
 \[0\ne (T^*_\tau \xi_{m,x}\mid T^*_\sigma \xi_{n,y})=(\xi_{m-d(\tau)_J,x(d(\tau),\infty)}\mid \xi_{n-d(\sigma)_J,y(d(\sigma),\infty)}).
\]
So $\xi_{m-d(\tau)_J,x(d(\tau),\infty)}=\xi_{n-d(\sigma)_J,y(d(\sigma),\infty)}$, $m-d(\tau)_J=n-d(\sigma)_J$, and there exists a semi-infinite path $x'$ such that $x=\tau x'$ and $y=\sigma x'$.
Since $\mu\sigma=\nu\tau$, we get $\nu x=\nu\tau x'=\mu\sigma x'=\mu y$. Also
$m-d(\tau)_J=n-d(\sigma)_J$ implies that $m+d(\nu)_J=n+d(\mu)_J$.
Therefore $(\xi_{m+d(\nu)_J,\nu x}\mid \xi_{n+d(\mu)_J,\mu y})=1$, which contradicts $(T^*_\mu T_\nu\xi_{m,x}\mid \xi_{n,y})=0$. Thus we have proved our third case, and we have verified that (T5) holds.

To see (T4), let $v\in\Lambda^0$ and $n\in\NN^k$.  Let  $\lambda,\mu\in v\Lambda^n$ such that $\lambda\ne \mu$. Since $\lambda$ and $\mu$ have the same degree they cannot have a common extension, and hence $\Lambda^{\text{min}}(\lambda,\mu)=\emptyset$. Then (T5) forces $T_\lambda (T^*_\lambda T_\mu )T^*_\mu=0$.
By (T2), $T_v T_\lambda T^*_\lambda=T_\lambda T^*_\lambda$, that is, $T_v \ge T_\lambda T^*_\lambda$. Thus $T_v\ge \sum_{\lambda\in v\Lambda^n}T_\lambda T^*_\lambda$ which is (T4). We have now shown that $\{T_\lambda:\lambda\in\Lambda\}$ is a Toeplitz-Cuntz-Krieger  $\Lambda$-family.

Finally, let $v\in\Lambda^0$,  $i\in K$ and  $x\in\partial^K\!\Lambda$. Let $e\in v\Lambda^{e_i}$. If $r(x)\neq v$, then $T_eT_e^*\xi_{m,x}=0=T_v\xi_{m,x}$, and hence
$(\sum_{f\in v\Lambda^{e_i}}T_fT_f^*)\xi_{m,x}=0=T_v\xi_{m,x}$.
So suppose $r(x)=v$. Since $d(e)_J=0$, 
\[\begin{split}
T_e T^*_e \xi_{m,x}&=
\begin{cases}T_e \xi_{m, x(d(e),\infty)}\quad\text{if}\;\;x(0,e_i)=e\\ 0\quad\quad\quad\quad\quad\quad\; \text{otherwise}\end{cases}\\
&=\begin{cases}\xi_{m,x}\quad\text{if}\;\;x(0,e_i)=e\\ 0\quad\quad\; \text{otherwise.}\end{cases}
\end{split}\]
Since $d(x)_i=\infty$ we have $T_eT_e^*\xi_{m,x}=\xi_{m,x}$ for exactly one edge $e=x(0,e_i)$, and hence $\big(\sum_{f\in v\Lambda^{e_i}}T_f T^*_f\big)\xi_{m,x}=T_v\xi_{m,x}$. Thus $\sum_{f\in v\Lambda^{e_i}}T_f T^*_f=T_v$, and we are done.
\end{proof}

\section{$\KMS$ states on Toeplitz algebras}\label{sec-kms-states}

For $r\in(0,\infty)^k$, define $\alpha^r:\RR\to\Aut \mathcal{T}C^*(\Lambda)$ in terms of the gauge action by $\alpha^r_t=\gamma_{e^{itr}}$. Then for $\mu,\nu\in \Lambda$, 
\[
\alpha^r_t(t_\mu t_\nu^*)=e^{itr\cdot(d(\mu)-d(\nu))}t_\mu t_\nu^* \]
is the restriction of the analytic function $z\mapsto e^{izr\cdot(d(\mu)-d(\nu))}t_\mu t_\nu^*$. Thus to see that a state is a KMS$_\beta$ state for $(\Tt C^*(\Lambda),\alpha^r)$, it suffices to check the $\KMS$ condition on pairs of elements of the form $t_\mu t_\nu^*$.

The following result is an improvement on \cite[Corollary~4.3]{aHLRS2}, which requires $\Lambda$ to be coordinatewise irreducible.

\begin{prop}\label{KMS1}
Let $\Lambda$ be a finite $k$-graph with no sources, and suppose that all the coordinate graphs $(\Lambda^0,\Lambda^{e_i},r,s)$ have cycles. Let $r\in (0,\infty)^k$ and $\beta\in [0,\infty)$. If there is a $\KMS_\beta$ state of  $(\TC^*(\Lambda), \alpha^r)$, then $\beta r_i \ge \ln \rho(A_i)$ for $1\leq i\leq k$.
\end{prop}

\begin{proof}
Suppose that $\phi$ is a KMS$_\beta$ state of $(\TC^*(\Lambda), \alpha^r)$, and fix $i$. Since $(\Lambda^0,\Lambda^{e_i},r,s)$ is not a cycle, there is a strongly connected component $C$ of this coordinate graph such that the $C\times C$ block $A_C$ of $A_i$ has $\rho(A_C)=\rho(A_i)>0$. (To see this, consider a Seneta decomposition of $A_i$, as described in \cite[\S2.3]{aHLRSseq}.) Then $A_C$ is irreducible, and the directed graph $E_C=(\Lambda^0,C\Lambda^{e_i}C,r,s)$ is strongly connected with vertex matrix $A_C$. The set $\{t_v:v\in C\}\cup\{t_e:e\in C\Lambda^{e_i}C\}$ is a Toeplitz-Cuntz-Krieger $E_C$-family in $\Tt C^*(\Lambda)$, and hence there is a homomorphism $\pi$ of $\Tt C^*(E_C)$ into $\Tt C^*(\Lambda)$ such that $\pi(s_e)=t_e$ for all $e\in C\Lambda^{e_i}C$. This homomorphism is equivariant for the action $\alpha$ studied in \cite{aHLRS1} and the action $\alpha'$ on $\Tt C^*(\Lambda)$ defined by $\alpha'_t=\alpha^r_{r_i^{-1}t}$. Lemma~\ref{scaledynamics} implies that $\phi$  is a KMS$_{r_i\beta}$ state of $(\Tt C^*(\Lambda),\alpha')$. Thus $\pi\circ\phi$ is a KMS$_{r_i\beta}$ state of $(\Tt C^*(E_C),\alpha)$, and since $E_C$ is strongly connected, it follows from \cite[Theorem~4.3(c)]{aHLRS1} that $r_i\beta\geq \ln\rho(A_C)=\ln\rho(A_i)$. 
\end{proof}

When $\beta$ is strictly larger than all the numbers $r_i^{-1}\ln \rho(A_i)$, Theorem~6.1 of \cite{aHLRS2} gives a $(|\Lambda^0|-1)$-dimensional simplex of KMS$_\beta$ states of $(\Tt C^*(\Lambda),\alpha^r)$. When $\beta$ is strictly less than any of the $r_i^{-1}\ln\rho(A_i)$, Proposition~\ref{KMS1} implies that there are no KMS$_\beta$ states at all. So the behaviour of the KMS$_\beta$ states changes dramatically as the inverse temperature  $\beta$ passes through the value
\[\beta_c:=\max_i\{r_i^{-1}\ln\rho(A_i)\};\]
we call $\beta_c$ the \emph{critical inverse temperature}. In this paper, we are interested in what happens at $\beta=\beta_c$.

Recall from Lemma~\ref{scaledynamics} that scaling the time $t$ does not effectively change the behaviour of KMS states. So in our case, replacing the vector $r$ by a scalar multiple will not change things significantly. We choose to use the unique multiple that has
\begin{equation}\label{choice}
\beta_c=\max_i\{r_i^{-1}\ln\rho(A_i)\}=1,
\end{equation}
and then we are interested in the KMS$_1$ states. To emphasise: even if we forget to say so locally, the restriction \eqref{choice} is in force throughout the rest of the paper. Thus we have $r_i=\ln\rho(A_i)$ for $i$ in some nonempty subset $K$ of $\{1,\dots,k\}$, and $r_i>\ln\rho(A_i)$ for $i\in J:=\{1,\dots,k\}\setminus K$. For the preferred dynamics studied in \cite[\S7]{aHLRS2} and \cite{aHLRS3}, we have $K=\{1,\dots,k\}$, but here we are thinking primarily about the case where $K$ is a proper subset of $\{1,\dots,k\}$.

\begin{prop}\label{KMS_combine}
Suppose that $\Lambda$ is a finite $k$-graph with no sources, and that $r\in (0,\infty)^k$. We suppose that $ r_i\geq\ln \rho(A_i)$ for all $i$, and that
\begin{equation}\label{defK}
K:=\big\{i\in\{1,\dots, k\}: r_i=\ln\rho(A_i)\big\}
\end{equation}
is nonempty.
\begin{enumerate}
\item\label{4a} There exist KMS$_1$ states of $(\TC^*(\Lambda),\alpha^r)$.

\item\label{4b} Every KMS$_1$ state  of $(\TC^*(\Lambda),\alpha^r)$ factors through the ideal generated by
\[
\Big\{ t_v-\sum_{e\in v\Lambda^{e_i}}t_e t^*_e : v\in\Lambda^0, i\in K, \text{$A_i$ is irreducible}\Big\}.
\]
\item\label{4c} Suppose that the coordinates of $r$ are rationally independent and that there exists $i\in K$ such that $A_i$ is irreducible. Let $\PF$ be the unimodular Perron-Frobenius eigenvector of $A_i$. Then there is a unique KMS$_1$ state $\phi$  of $(\TC^*(\Lambda),\alpha^r)$, and
\begin{equation*}%eq:TCK\label{KMS_combine-eq}
\phi(t_\mu t_\nu^*)=\delta_{\mu, \nu}e^{-r\cdot d(\mu)}\PF_{s(\mu)}\quad\text{for $\mu,\nu\in\Lambda$.}
\end{equation*}
\end{enumerate}
\end{prop}

\begin{proof}
Choose a decreasing sequence $\{\beta_n\}\subset (1, \infty)$ such that $\beta_n\to 1$. Since $\beta_n> 1 \geq r_i^{-1}\ln \rho(A_i)$ for all $i$, \cite[Theorem~6.1]{aHLRS2} implies that there is at least one $\KMS_{\beta_n}$ state $\phi_n$ of $(\TC^*(\Lambda),\alpha^r)$. Since the state space of $\TC^*(\Lambda)$ is $\text{weak}^*$ compact, we may assume by passing to a subsequence that $\phi_n\to \phi$. Then $\phi$ is a $\KMS_{1}$ state of $(\TC^*(\Lambda),\alpha^r)$ by \cite[Proposition~5.3.23]{BR}. This gives~\eqref{4a}.

For~\eqref{4b}, suppose $i\in K$ and  $A_i$ is irreducible.  Let $\phi$ be a $\KMS_1$ state, and for $v\in\Lambda^0$, set $m^\phi_v=\phi(t_v)$. Then \cite[Proposition~4.1]{aHLRS2} implies that $m^\phi\in [0,1]^{\Lambda^0}$ is a probability measure on $\Lambda^0$ such that $(1- e^{-r_i}A_i)m^\phi \ge 0$. Since $r_i=\ln \rho(A_i)$, we get $A_i m^\phi\le e^{r_i}m^\phi= \rho(A_i)m^\phi$. Since $A_i$ is irreducible and $\rho(A_i)$ is the Perron-Frobenius eigenvalue of $A_i$, the subinvariance theorem \cite[Theorem~1.6]{Sen} implies that $A_i m^\phi=\rho(A_i)m^\phi$. Thus  $m^\phi$ is the unimodular Perron-Frobenius eigenvector  of $A_i$. Now
\begin{align*}
\phi\Big(\sum_{e\in v\Lambda^{e_i}}t_e t^*_e\Big)&=\sum_{e\in v\Lambda^{e_i}} e^{-r_i}\phi(t_{s(e)})
=\rho(A_i)^{-1}\sum_{w\in\Lambda^0}|v\Lambda^{e_i}w|\phi(t_w)\\
&=\rho(A_i)^{-1}(A_i m^{\phi})_v=m^\phi_v=\phi(t_v),
\end{align*}
and \eqref{4b} follows from the general lemma \cite[Lemma~2.2]{aHLRS1}.

For~\eqref{4c}, suppose $\phi$  and $\psi$ are KMS$_1$ states. Let $i\in K$ such that $A_i$ is irreducible. The argument of \eqref{4b} then implies that  $m^\phi=m^\psi$ is  the unimodular Perron-Frobenius eigenvector $\PF$ of $A_i$.
 Since $r$ has rationally independent coordinates, \cite[Proposition~3.1(b)]{aHLRS2} says that for all $\mu,\nu\in \Lambda$, we have
\[
\phi(t_\mu t_\nu^*)=\delta_{\mu, \nu}e^{-r\cdot d(\mu)}\phi(t_{s(\mu)})=\delta_{\mu, \nu}e^{- r\cdot d(\mu)}\PF_{s(\mu)}=\psi(t_\mu t_\nu^*),
\]
which implies that $\phi=\psi$.
\end{proof}

\begin{rmk}\label{atomicmeas}
Above the critical inverse temperature, it is quite easy to find a spatial implementation of the sort we seek using the calculations in \cite{aHLRS2}. To see this, suppose that $\Lambda$ is a finite $k$-graph with no sources and $\beta>\beta_c=1$. Let $(T,Q)$ be the path representation of $\Lambda$ on $\ell^2(\Lambda)$ described at the end of \cite[\S2.2]{aHLRS2}, and $\pi_{T,Q}$ the corresponding representation of $\Tt C^*(\Lambda)$. Write $\{h_\lambda:\lambda\in \Lambda\}$ for the usual orthonormal basis of $\ell^2(\Lambda)$. Then \cite[Theorem~6.1]{aHLRS2} describes the KMS$_\beta$ states of $(\Tt C^*(\Lambda),\alpha^r)$, as follows: associated to each $\epsilon\in[0,\infty)^{\Lambda^0}$ satisfying a constraint $\epsilon\cdot y=1$, there  is a KMS$_\beta$ state $\phi_\epsilon$ such that
\[
\phi_\epsilon(a)=\sum_{\lambda\in \Lambda} (\pi_{T,Q}(a)h_\lambda\,|\,h_\lambda ) e^{-\beta r\cdot d(\lambda)}\epsilon_{s(\lambda)}.
\]
(See \cite[page~279]{aHLRS2}, where the weight $e^{-\beta r\cdot d(\lambda)}\epsilon_{s(\lambda)}$ was denoted $\Delta_\lambda$.) We now define a measure $\nu_\beta$ on $\Lambda$ by $\nu_\beta(\{\lambda\})=e^{-\beta r\cdot d(\lambda)}\epsilon_{s(\lambda)}$. Since $\phi_\epsilon(1)=1$, $\nu_\beta$ is a probability measure, and then 
\begin{equation}\label{spatiallarge}
\phi_\epsilon(a)=\int_\Lambda (\pi_{T,Q}(a)h_\lambda\,|\,h_\lambda )\,d\nu_\beta(\lambda).
\end{equation}
\end{rmk}

\begin{rmk}\label{nolimit}
The spatial realisation \eqref{spatiallarge} breaks down at $\beta=\beta_c$. We illustrate the problems by  considering the case $k=1$. For this brief discussion, we resume the notation of \cite{aHLRS1}, where the dynamics is normalised to give $\beta_c=\ln\rho(A)$. As $\beta$ decreases to $\beta_c$, the factor $e^{-\beta |\lambda|}$ converges to $\rho(A)^{-|\lambda|}$, so at first sight the measures $\nu_\beta$ converge. However, the constraint $\epsilon\cdot y=1$ involves the vector $y$ of \cite[Theorem~3.1(a)]{aHLRS1}, which depends on $\beta$. In particular, the $\epsilon$ which is a multiple of the point mass $\delta_v$ is $y_v^{-1}\delta_v$. The argument in the first paragraph of the proof of \cite[Theorem~3.1(a)]{aHLRS1}, and in particular the calculation (3.2), shows that the convergence of the series defining $y_v$ is equivalent to that of the series $\sum_{n=0}^\infty e^{-\beta n}A^n$. But this series does not converge for $\beta=\ln \rho(A)$: if it did, it would give an inverse for $1-\rho(A)^{-1}A$, which is not invertible because $\rho(A)$ is an eigenvalue of $A$ (by Perron-Frobenius theory\footnote{Applied to an irreducible block of $A$ if $A$ is not itself irreducible.}). So there is no multiple of $\delta_v$ to which we can apply \cite[Theorem~3.1(b)]{aHLRS1}. 

Geometrically, the simplex $\Sigma_\beta$ shrinks towards the origin as $\beta\to \beta_c$, and hence the measures $\nu_\beta$ satisfy $\nu_\beta(\{\lambda\})\to 0$ for each fixed $\lambda$. So as $\beta\to \beta_c$,  the mass distribution of the probability measure $\nu_\beta$ is spreading out. 

We can still get a KMS$_{\beta_c}$ state of $\Tt C^*(E)$ by taking limits of KMS$_\beta$ states as $\beta\to \beta_c+$, but this state is not spatially realisable on $\ell^2(E^*)$. Indeed, if $A$ is irreducible, this state factors through the quotient map $\Tt C^*(E)\to C^*(E)$ \cite[Theorem~4.3]{aHLRS1}, and the usual faithful representation of $C^*(E)$ is on the infinite-path space $\ell^2(E^\infty)$. (The case $k=1$ is different from $k>1$: the spectrum of the commutative subalgebra $D$ discussed in the introduction is $E^*\cup E^\infty$ \cite{W2}.)
\end{rmk}

\section{A spatial realisation of a $\KMS$ state}\label{sec-spatial}

We summarise our main results as follows. The map $\Phi^\gamma$ appearing in \eqref{eq-formula-phi} is the expectation of $\Tt C^*(\Lambda)$ onto the fixed-point algebra $\Tt C^*(\Lambda)^\gamma$ obtained by averaging over the gauge action $\gamma$ of $\TT^k$.

\begin{thm}\label{main_thm}
Let $\Lambda$ be a finite coordinatewise-irreducible $k$-graph with no sources, and let $\PF$ be the common unimodular Perron-Frobenius eigenvector of the vertex matrices $A_i$. Let $J\sqcup K$ be a nontrivial partition of $\{1,\dots, k\}$, and suppose that $r\in (0,\infty)^k$ satisfies
\[
 r_j > \ln \rho(A_j) \;\;\text{for}\;\;j\in J\;\;\text{and}\;\; r_i= \ln\rho(A_i) \;\;\text{for}\;\; i\in K.
 \]
Set $C_J:=\prod_{j\in J}(1-e^{- r_j}\rho(A_j))$.
\begin{enumerate}
\item\label{6a} For each $m\in \N^J$ there is a measure $\nu^m$ on $\Lambda^{m,\infty_K}$ such that, for $n\in \NN^K$ and $\lambda\in \Lambda^{m,n}$,
\begin{equation}\label{nu_m} 
\nu^m(Z(\lambda)\cap \Lambda^{m,\infty_K})=e^{-r\cdot d(\lambda)}C_J\PF_{s(\lambda)}=e^{-r\cdot (m,n)}C_J\PF_{s(\lambda)}.
\end{equation}
\item\label{6b} Let $m\in\NN^J$, $n\in\NN^K$ and $\lambda\in\Lambda^{m,n}$. Then
 \begin{equation}\label{eq-41}
 \sum_{l\in \NN^J}\nu^l(Z(\lambda)\cap \Lambda^{l,\infty_K})=\sum_{l\ge m}\nu^l(Z(\lambda)\cap \Lambda^{l,\infty_K})=e^{-r\cdot d(\lambda)}\PF_{s(\lambda)}.
 \end{equation}

\item\label{6c} Let $\pi^K:\mathcal{T}C^*(\Lambda)\to B(\ell^2(\partial^K\!\Lambda))$ be the semi-infinite path representation of \S\ref{semiinf}. Then there is a bounded functional $\phi$ on $\TC^*(\Lambda)$ such that
\begin{equation}\label{eq-formula-phi}
\phi(a):=\sum_{m\in\NN^J}\int (\pi^K(\Phi^\gamma(a))\xi_{m,x}\mid \xi_{m,x})\, d\nu^m(x)\quad\text{for $a\geq 0$,}
 \end{equation}
and  $\phi$  is a $\KMS_1$ state of $(\TC^*(\Lambda),\alpha^r)$ satisfying 
\begin{equation}\label{kmsformula}
 \phi(t_\sigma t^*_\tau)=\delta_{\sigma,\tau}e^{-r\cdot d(\sigma)}\PF_{s(\sigma)}\quad\text{for $\sigma,\tau\in \Lambda$.}
 \end{equation}
The state $\phi$ factors through the quotient by the ideal  generated by
\[
\Big\{t_v-\sum_{e\in v\Lambda^{e_i}}t_e t_e^*\;: \; v\in \Lambda^0, i\in K\Big\},
\]
and we have $\phi(t_v-\sum_{e\in v\Lambda^{e_j}}t_e t_e^*)\ne 0$ for all $j\in J$ and $v\in\Lambda^0$.
 \item\label{6d} If $r$ has rationally independent coordinates, then the state $\phi$ of \eqref{6c} is the only $\KMS_1$ state of $(\TC^*(\Lambda), \alpha^r)$.
\end{enumerate}
\end{thm}

Before we start the proof of Theorem~\ref{main_thm}, we prove a couple of lemmas. The next lemma describes a standard construction of measures on inverse limits. It is a mild generalisation of \cite[Lemma~6.1]{BLPRR} (where the partially ordered set is $\NN$ and the measures are probability measures), and the proof given there carries over.

\begin{lem}\label{lem:measure}
Let $I$ be a directed partially ordered set with smallest element $0$. For $i,j\in I$ let  $X_i$ be a compact space and $r_{ij}:X_j\to X_i$ be a surjection.
Let $(X_\infty, \pi_i)$ be the inverse limit of the system $(\{X_i\}, \{r_{ij}\})_{i,j\in I}$.  Suppose that we have Borel measures $\mu_i$ on $X_i$ such that $\mu_0$ is finite and
\[
\int (f\circ r_{ij})\, d\mu_j=\int f \, d\mu_i\quad\text{for $i\le j$ and $f\in C(X_i)$.}
\]
Then there is a unique finite Borel measure $\mu$ on $X_\infty$ such that
\begin{equation*}%\label{eq-measure}
\int f\circ \pi_i\, d\mu=\int f \, d\mu_i\quad\text{for $f\in C(X_i)$.}
\end{equation*}
\end{lem}

\begin{lem}\label{lem:measurability'}
Let $\Lambda$ be a finite $k$-graph with no sources and $J\sqcup K$ be a nontrivial partition of $\{1,\dots,k\}$. Let $\pi^K$ be the semi-infinite path representation of $\TC^*(\Lambda)$ from \S\ref{semiinf}. Let $\sigma,\tau\in\Lambda$ and $m\in\NN^J$.
Then for $x\in \Lambda^{m,\infty_K}$, we have
\[
\big(\pi^K(t_\sigma t^*_\tau) \xi_{m,x} \mid \xi_{m,x}\big)=\delta_{d(\tau)_J,d(\sigma)_J}\chi_{Z(\tau)\cap Z(\sigma)\cap \Lambda^{m,\infty_K}}(x).
\]
Let $f:\partial^K\!\Lambda\to \RR$ be the function defined by $f(x)=(\pi^K(t_\sigma t^*_\tau) \xi_{m,x} \mid \xi_{m,x})$ for $x\in \Lambda^{m,\infty_K}$. 
Then $f$ is Borel and its restriction to $\Lambda^{m,\infty_K}$ is continuous.
\end{lem}

\begin{proof}
Let $x\in \Lambda^{m,\infty_K}$. Then
\[\begin{split}
(\pi^K (t_\sigma t^*_\tau) \xi_{m,x}\mid\xi_{m,x})
&=(T^*_\tau \xi_{m,x} \mid T^*_\sigma \xi_{m,x})\\
&=(\xi_{m-d(\tau)_J, x(d(\tau),\infty)} \mid \xi_{m-d(\sigma)_J, x(d(\sigma),\infty)})\\
&=\begin{cases}1\quad\text{if}\;\; d(\tau)_J=d(\sigma)_J\;\;\text{and}\;\; x\in Z(\tau)\cap Z(\sigma)\cap\Lambda^{m,\infty_K}\\
0\quad\text{otherwise}.\end{cases}\\
&=\delta_{d(\tau)_J,d(\sigma)_J}\chi_{Z(\tau)\cap Z(\sigma)\cap\Lambda^{m,\infty_K}}(x).
\end{split}\]
So either $f=0$ or $f$  is the characteristic function of the Borel set $Z(\tau)\cap Z(\sigma)\cap\partial^K\!\Lambda$. In either case, $f$ is Borel. 

By Proposition~\ref{idinvlim}, $\Lambda^{m,\infty_K}$ is compact. 
Since $Z(\tau)\cap Z(\sigma)$ is compact and open in $W_\Lambda$ by Lemma~\ref{oversight}, its intersection with the compact set $\Lambda^{m,\infty_K}$ is compact and open in $\Lambda^{m,\infty_K}$. So the restriction of $f$ to $\Lambda^{m,\infty_K}$ is continuous.
\end{proof}

\begin{proof}[Proof of Theorem~\ref{main_thm}]
We construct the measure $\nu^m$ using Lemma~\ref{lem:measure}. For $n\in \NN^K$, we give $\Lambda^{m,n}$ the discrete topology, and let $\nu^{m,n}$ be the measure on $\Lambda^{m,n}$ such that
\[
\nu^{m,n}(\{\lambda\})=e^{-r\cdot (m,n)}C_J\PF_{s(\lambda)}\quad\text{for $\lambda\in \Lambda^{m,n}$.}
\]
For $p,n\in\NN^K$ such that $p\ge n$, we define $r_{n,p}:\Lambda^{m,p}\to \Lambda^{m,n}$ by $r_{n,p}(\lambda)=\lambda(0, (m,n))$; since $\Lambda$ has no sources, each $r_{n,p}$ is a surjection.

We claim that $\int f\circ r_{n,p}\,d\nu^{m,p}=\int f \,d\nu^{m,n}$ for all $f\in C(\Lambda^{m,n})$.
Since the characteristic functions of singletons span $C(\Lambda^{m,n})$, it suffices to consider $f=\chi_{\{\lambda\}}$.
A quick calculation shows that
\[
\chi_{\{\lambda\}}\circ r_{n,p}=\sum_{\alpha\in s(\lambda)\Lambda^{0,p-n}}\chi_{\{\lambda\alpha\}},
\]
and hence
\begin{align}\label{calcusingPF}
\int \chi_{\{\lambda\}}\circ r_{n,p}\, d\nu^{m,p}&=\sum_{\alpha\in s(\lambda)\Lambda^{0,p-n}}\nu^{m,p}(\{\lambda\alpha\})\\
&=e^{- r\cdot(m,p)}C_J\sum_{w\in\Lambda^0} A^{(0,p-n)}(s(\lambda),w) \PF_w\notag\\
&=e^{- r\cdot(m,p)} C_J\rho(\Lambda)^{(0,p-n)}\PF_{s(\lambda)}\notag\\
&=e^{- r\cdot(m,p)}e^{r\cdot (0,p-n)} C_J\PF_{s(\lambda)}\quad\text{(since $e^{r_i}=\rho(A_i)$ for $i\in K$)}\notag\\
&=e^{- r\cdot(m,n)}C_J\PF_{s(\lambda)}\notag\\
&=\int \chi_{\{\lambda\}}\, d\nu^{m,n},\notag
\end{align}
as claimed.

Since $\Lambda^{m,0}$ is finite, $\nu^{m,0}$ is trivially a finite measure, and Lemma~\ref{lem:measure} gives a unique measure $\nu^m$ on $\Lambda^{m,\infty_K}$ such that,  for $\lambda\in \Lambda^{m,n}$,
\begin{align*}%\label{eq-measurecalculation2}
\nu^m(Z(\lambda)\cap \Lambda^{m,\infty_K})
&=\int\chi_{\{\lambda\}}\circ \pi_n\, d\nu^m
=\int\chi _{\{\lambda\}}\, d\nu^{m,n}\\
&=\nu^{m,n}(\{\lambda\})
= e^{- r\cdot d(\lambda)}C_J\PF_{s(\lambda)},\notag
\end{align*}
and we have proved \eqref{6a}.

For \eqref{6b}, we fix $m\in\NN^J$, $n\in\NN^K$ and $\lambda\in\Lambda^{m,n}$. We first observe that for $l\in \NN^J$, we have
\[ 
Z(\lambda)\cap\Lambda^{l,\infty_K}\not=\emptyset\Longrightarrow l\geq m\text{ in $\NN^J$,}
\]
and this immediately implies the first equality in \eqref{eq-41}. So we suppose that $l\geq m$ in $\NN^J$. Then repeating the first few steps in the calculation \eqref{calcusingPF} gives
\begin{align}\label{compnul}
\nu^l(Z(\lambda)\cap\Lambda^{l,\infty_K})
&=\sum_{\alpha\in s(\lambda)\Lambda^{l-m,0}}\nu^l(Z(\lambda\alpha)\cap\Lambda^{l,\infty_K})\\
&=e^{- r\cdot(l,n)}C_J\big(A^{(l-m,0)}\PF\big)_{s(\lambda)}\notag\\
&=e^{- r\cdot(l,n)}C_J\rho(\Lambda)^{(l-m,0)}\PF_{s(\lambda)}.\notag
\end{align}
Summing over $l\in \NN^J$, writing $l=m+p$ and remembering that $e^{-r_j}\rho(A_j)<1$ for $j\in J$ gives
\begin{align*}
\sum_{l\in\NN^J}\nu^l(Z(\lambda)\cap\Lambda^{l,\infty_K})
&=e^{- r\cdot(m, n)}C_J\Big(\sum_{p\in\NN^J}e^{- r\cdot(p, 0)}\rho(\Lambda)^{(p,0)}\big)\PF_{s(\lambda)}\\
&=e^{-r\cdot d(\lambda)}C_J\Big(\prod_{j\in J}(1-e^{-r_j}\rho(A_j))^{-1}\Big)\PF_{s(\lambda)}\\
&=e^{-r\cdot d(\lambda)}\PF_{s(\lambda)}.
\end{align*}
This gives \eqref{6b}. 

For \eqref{6c},  we consider  $a\in\TC^*(\Lambda)$.  Lemma~\ref{lem:measurability'} implies that $x\mapsto (\pi^K(a)\xi_{m,x}\mid \xi_{m,x})$ is continuous on $\Lambda^{m,\infty_K}$, so the integrals make sense, and we next have to show that the sum on the right-hand side of \eqref{eq-formula-phi} converges absolutely. For $m\in \NN^J$ we compute
\[
\nu^m(\Lambda^{m,\infty_K})=\sum_{\lambda\in \Lambda^{m,0}}\nu^m(Z(\lambda)\cap \Lambda^{m,\infty_K})
\]
using \eqref{compnul} with $l=m$ as well as the techniques of \eqref{compnul}, finding that 
\[
\nu^m(\Lambda^{m,\infty_K})=e^{-r\cdot (m,0)}\rho(\Lambda)^{(m,0)}C_J.
\]
Since $j\in J$ implies $1>e^{-r_j}rho(A_j)$ we get $\sum_{m\in \NN^J}e^{-r\cdot (m,0)}\rho(\Lambda)^{(m,0)}C_J=1$. Since the integrands all have absolute value at most $\|a\|$, the series in \eqref{eq-formula-phi} converges absolutely, with sum at most $\|a\|$. So there is a functional $\phi$ satisfying \eqref{eq-formula-phi}, and this functional has norm at most one. For $a\geq 0$ all the summands in \eqref{eq-formula-phi} are non-negative, and hence $\phi$ is positive. Equation~\eqref{eq-41} implies that $\phi(1)=1$, and hence $\phi$ is a state.  

Next we show that $\phi$ satisfies \eqref{kmsformula}. 
Fix $\sigma\in\Lambda$ and $x\in\Lambda^{m,\infty_K}$. If $m\ge d(\sigma)_J$, then Lemma~\ref{lem:measurability'} gives
\[
(\pi^K(t_\sigma t^*_\sigma)\xi_{m,x}\mid \xi_{m,x})
=\chi_{Z(\sigma)\cap\Lambda^{m,\infty_K}}(x);
\]
otherwise, it is $0$. Thus \eqref{6b} gives half of \eqref{kmsformula}:
\[
\phi(t_\sigma t^*_\sigma)=\sum_{m \ge d(\sigma)_J}\nu^m(Z(\sigma)\cap \Lambda^{m,\infty_K})=e^{- r\cdot d(\sigma)} \PF_{s(\sigma)}.
\] 

Now take a pair $\sigma, \tau\in \Lambda$. If $d(\sigma)\not=d(\tau)$, then $\Phi^\gamma(t_\sigma t^*_\tau)=0$ and $\phi(t_\sigma t^*_\tau)=0$. So suppose that $d(\sigma)=d(\tau)$, $\sigma\not=\tau$ and $\phi(t_\sigma t^*_\tau)\not=0$. Then there exists $m\in \NN^J$ and $x\in \Lambda^{m,\infty_K}$ such that 
\[
\big(\pi^K(t_\sigma t^*_\tau) \xi_{m,x} \mid \xi_{m,x}\big)=\big(\pi^K(t_\tau)^*\xi_{m,x} \mid \pi^K(t_{\sigma})^*\xi_{m,x}\big)\not=0.
\]
But then there exists $y$ such that $x=\sigma y=\tau y$ and $\sigma=x(0,d(\sigma))=x(0,d(\tau))=\tau$. This gives the other half of \eqref{kmsformula}.

In view of \eqref{kmsformula}, it follows from \cite[Proposition~3.1(b)]{aHLRS2} that $\phi$ is a KMS$_1$ state and hence that $\phi$ is a $\KMS_1$ state of $(\TC^*(\Lambda),\alpha)$. Since $\Lambda$ is coordinatewise irreducible, all the $A_i$ are irreducible, and with $K=\{i\}$, Proposition~\ref{KMS_combine}\eqref{4b} implies that
$\phi$ factors through the ideal generated by $\{t_v-\sum_{e\in v\Lambda^{e_i}}t_e t^*_e:v\in\Lambda^0\}$.

To complete the proof of part \eqref{4c}, we take $v\in\Lambda^0$ and $j\in J$, and consider the value of $\phi$ on the projection $t_v-\sum_{e\in v\Lambda^{e_j}}t_et_e^*$. We compute, using \eqref{eq-41} and \eqref{kmsformula},
\begin{align*}
\phi\Big(t_v-\sum_{e\in v\Lambda^{e_j}}t_e t_e^*\Big)&=\phi(t_v)-\sum_{e\in v\Lambda^{e_j}}\phi(t_e t_e^*)\\
&=\sum_{m\in\NN^J}\nu^m(Z(v)\cap \Lambda^{m,\infty_K})-\sum_{e\in v\Lambda^{e_j}}\sum_{m\ge e_j}\nu^m(Z(e)\cap\Lambda^{m,\infty_K})\\
&=\sum_{m\in\NN^J, m_j=0}\nu^m(Z(v)\cap \Lambda^{m,\infty_K})+\sum_{m\in\NN^J, m\ge e_j}\nu^m(Z(v)\cap \Lambda^{m,\infty_K})\\
& \hskip5cm -\sum_{e\in v\Lambda^{e_j}}\sum_{m\ge e_j}\nu^m(Z(e)\cap\Lambda^{m,\infty_K})\\
&=\sum_{m\in\NN^J, m_j=0}\nu^m(Z(v)\cap \Lambda^{m,\infty_K})\\
&\ge \nu^0(Z(v)\cap \Lambda^{0,\infty_K})\\
&=C_J\PF_v\quad\text{by \eqref{nu_m}.}
\end{align*}
Since $\PF$ is the common Perron-Frobenius eigenvector for the $A_i$, it has positive entries, and \eqref{6c} follows.

Item~\eqref{6d} follows from Proposition~\ref{KMS_combine}\eqref{4c}.
\end{proof}

\begin{cor}\label{cor:measure} Let $\phi$ be the  $\KMS_1$ state of $(\TC^*(\Lambda),\alpha^r)$ from  Theorem \ref{main_thm}\eqref{6c}.
 Then there is a probability measure $\mu$ on $\partial^K\!\Lambda=\bigcup_{m\in\NN^J}\Lambda^{m,\infty_K}$ such that 
 \begin{equation}\label{eq:cor_measure}
 \phi(a)=\int \big(\pi^K(\Phi^\gamma(a))\xi_{m,x}\mid\xi_{m,x}\big)\,d\mu(x)\quad\text{for $a\in\TC^*(\Lambda)$}.
 \end{equation}
In particular, for $\sigma\in\Lambda^{m,n}$ we have
\begin{equation}\label{eq:cor_measure2}
\phi(t_\sigma t^*_\sigma)=\mu(Z(\sigma)\cap\partial^K\!\Lambda).
\end{equation}
\end{cor}

\begin{proof}
By Theorem~\ref{main_thm}\eqref{6a}, for $m\in\NN^J$ we have finite Borel measures $\nu^m$ on $\Lambda^{m,\infty_K}$ which we can view as measures on $\partial^K\!\Lambda=\bigcup_{m\in\NN^J}\Lambda^{m,\infty_K}$ with support in $\Lambda^{m,\infty_K}$. Then by Theorem~\ref{main_thm}\eqref{6b} we have
\begin{align}
\sum_{m\in\NN^J}\nu^m(\partial^K\!\Lambda)&=\sum_{m\in\NN^J}\nu^m(\Lambda^{m,\infty_K})=\sum_{m\in\NN^J}\sum_{v\in\Lambda^0}\nu^m(Z(v)\cap\Lambda^{m,\infty_K})\label{eq:sum of nu}\\
&=\sum_{v\in\Lambda^0}\sum_{m\in\NN^J}\nu^m(Z(v)\cap\Lambda^{m,\infty_K})\quad\text{(by Tonelli's Theorem)}\notag\\
&=\sum_{v\in\Lambda^0}\PF_v=1.\notag
\end{align}
Since we are viewing the $\nu^m$ as measures on $\partial^K\!\Lambda$, they define the functionals on $C(\partial^K\!\Lambda)$ with norm $\Vert \nu^m\Vert=\nu^m(\partial^K\!\Lambda)$. Thus \eqref{eq:sum of nu} implies that the series $\sum_{m\in\NN^J}\nu^m$ converges in $C(\partial^K\!\Lambda)^*$ to a positive functional on $C(\partial^K\!\Lambda)$ of norm $1$, which by the Riesz representation theorem is given by a probability measure $\mu$ on $\partial^K\!\Lambda$. Then for $f\in C(\partial^K\!\Lambda)$, we have
\begin{equation}\label{eq:mu}
\int f\, d\mu =\sum_{m\in\NN^J}\int f\, d\nu^m.
\end{equation}
Let $\sigma,\tau\in\Lambda$. Then Lemma~\ref{lem:measurability'} shows that  $x\mapsto (\pi^K(t_\sigma t^*_\tau) \xi_{m,x} \mid \xi_{m,x})$ is Borel on $\partial^K\!\Lambda$, and hence it is $\mu$-measurable. Then  \eqref{eq:mu} and 
the formula for $\phi$ in
\eqref{eq-formula-phi} give
\[\begin{split}
\int \big(\pi^K(t_\sigma t^*_\tau)\xi_{m,x} \mid \xi_{m,x}\big)\,d\mu(x)=\sum_{m\in\NN^J}\int \big(\pi^K(t_\sigma t^*_\tau)\xi_{m,x} \mid \xi_{m,x}\big)\, d\nu^m(x)=\phi(t_\sigma t^*_\tau).
\end{split}\]
Now \eqref{eq:cor_measure} follows by continuity and \eqref{eq:cor_measure2} follows by taking $\sigma=\tau$.
\end{proof}

\begin{rmk}
Our proof uses that $K$ is not all of $\{1,\dots,k\}$, so that $J$ is nonempty. A similar result for the case $J=\emptyset$ is proved in \cite[Proposition~10.2]{aHLRS3}. However, it is easier to construct the measure when $J=\emptyset$, because then we can view $\Lambda^\infty$ as the inverse limit $\varprojlim_{n\in \NN^k}\Lambda^n$ of the finite path spaces (see the proof of \cite[Proposition~8.1]{aHLRS3}). Note that \cite[Proposition~10.2]{aHLRS3} applies to a broader class of graphs.
\end{rmk}

\begin{rmk}\label{periodic?}
Theorem~7.1 of \cite{aHLRS3} says that the KMS$_1$ state is unique if and only if the graph is aperiodic. At least for graphs with one vertex, the ``rationally independent'' hypothesis in Theorem~\ref{main_thm}\eqref{6d} is linked to aperiodicity. If $\Lambda$ has one vertex, $N_1$ blue edges and $N_2$ red edges, and if $\ln N_1/\ln N_2$ is irrational, then $\Lambda$ is aperiodic \cite[Corollary~3.2]{DY}. However, even if $\ln N_1/\ln N_2$ is rational, then $\Lambda$ can be aperiodic. There is a detailed discussion of this question in \cite{DY}, and a necessary and sufficient condition is described in \cite[Theorem~3.1]{DY}. (The proof of this in \cite{DY} is algebraic: when $|\Lambda^0|=1$, the path space $\Lambda$ is a semigroup, and one can study the graph by studying the algebraic properties of this semigroup. There is an alternative graph-based proof in the appendix to \cite{BR2}.)
\end{rmk}

\section{Better results for $2$-graphs}\label{secnogaugeexpectation}

The formula~\eqref{eq-formula-phi} for the KMS state $\phi$ in Theorem~\ref{main_thm} involves the expectation $\Phi^\gamma$ onto the core $\Tt C^*(\Lambda)^\gamma$. It did not appear in the corresponding formula in \cite[Proposition~10.2]{aHLRS3}, so one naturally wonders whether it is necessary here. We have been able to answer this when $|K|=1$: formula \eqref{kmsformula2} implies that $\psi$ is the state $\phi$ of Theorem~\ref{main_thm}. (Of course, this is the only nontrivial possibility for $K$ when $k=2$.)  For $|K|\geq 2$, the state $\psi$ is not necessarily supported on the diagonal $\clsp\{t_\lambda t_\lambda^*\}$, and hence need not be the state in Theorem~\ref{main_thm} (see Remark~\ref{notunique}).

\begin{prop}\label{Ksingleton}
Resume the notation of Theorem~\ref{main_thm}, and suppose in addition that $K=\{i\}$ and that the directed graph $(\Lambda^0, \Lambda^{e_i}, r,s)$ is not a cycle.  Then there is a bounded functional $\psi$ on $\TC^*(\Lambda)$ such that
\begin{equation}\label{eq-formula-phi2}
\psi(a):=\sum_{m\in\NN^J}\int_{\Lambda^{m,\infty_K}} \big(\pi^K(a)\xi_{m,x}\mid \xi_{m,x}\big)\, d\nu^m(x)\quad\text{for $a\geq 0$,}
 \end{equation}
and  $\psi$  is a $\KMS_1$ state of $(\TC^*(\Lambda),\alpha^r)$ satisfying 
\begin{equation}\label{kmsformula2}
 \psi(t_\sigma t^*_\tau)=\delta_{\sigma,\tau}e^{-r\cdot d(\sigma)}\PF_{s(\sigma)}\quad\text{for $\sigma,\tau\in \Lambda$.}
 \end{equation}
The state $\psi$ factors through the quotient by the ideal  generated by $t_v-\sum_{e\in v\Lambda^{e_i}}t_e t_e^*$,
and we have $\psi\big(t_v-\sum_{e\in v\Lambda^{e_j}}t_e t_e^*\big)\ne 0$ for all $j\in J$ and $v\in\Lambda^0$.
\end{prop}

\begin{proof}
The first two paragraphs of the proof of Theorem~\ref{main_thm}\eqref{6c} apply almost verbatim. 
For the other half of \eqref{kmsformula2}, we fix $\tau\neq \sigma \in \Lambda$, and aim to get an estimate for $\psi(t_\sigma t_\tau^*)$. First we have to look at the integrands. For $a$ of the form $t_\sigma t_\tau^*$ they all take only the values $0$ and $1$. We need to fix $m$ and look at the semi-infinite paths $x\in \Lambda^{m,\infty_K}$ such that
\[
\big(\pi^K(t_\sigma t^*_\tau)\xi_{m,x}\mid\xi_{m,x}\big)=\big(\xi_{m-d(\tau)_J,x(d(\tau),\infty)}\mid\xi_{m-d(\sigma)_J,x(d(\sigma),\infty)}\big)=1.
\] 
If $x$ is such a path, then $d(\tau)_J=d(\sigma)_J\leq m$,  and  there is a semi-infinite path $y$ such that $x=\tau y=\sigma y$.  Say $d(\tau)=(d(\tau)_J,n)$ and $d(\sigma)=(d(\tau)_J,p)$ for some $n, p\in\NN^K$. We cannot have $p=n$, because then $d(\tau)=d(\sigma)$ and $\sigma y=\tau y$ imply $\sigma=\tau$. Since $|K|=1$, $\NN^K\cong \NN$ is totally ordered, so we may suppose that $p> n$ (otherwise swap $p$ and $n$).
Then $\sigma y=\tau y$ implies that $\sigma(0,(d(\tau)_J,n))=\tau$. So $\sigma y=\tau \lambda y$ where $\lambda=\sigma ((d(\tau)_J,n),(d(\tau)_J,p))$.
Thus $y=\lambda y$, and $y=\lambda\lambda y$. By induction we have $y=\lambda^Ny$ for all $N\geq 1$, and $x\in Z(\tau \lambda^N)\cap \Lambda^{m,\infty_K}$ for all $N\ge 1$.
Thus 
\[
\big\{x\in\Lambda^{m,\infty_K}:\big(\pi^K (t_\sigma t^*_\tau)\xi_{m,x}\mid\xi_{m,x}\big)=1\big\}\subset \textstyle{\bigcap_{N=1}^{\infty}}Z(\tau\lambda^N) \cap \Lambda^{m,\infty_K},
\] 
and we can estimate
\begin{align}\label{estphi}
\psi(t_\sigma t^*_\tau)&=\sum_{m\in\NN^J}\nu^m\big(\big\{x\in\Lambda^{m,\infty_K}:(\pi^K(t_\sigma t^*_\tau)\xi_{m,x}\mid\xi_{m,x})=1\big\}\big)\\
&\le \sum_{m\geq d(\tau)_J}\nu^m\big(\textstyle{\bigcap_{N=1}^{\infty}}Z(\tau\lambda^N)\cap \Lambda^{m,\infty_K}\big)\notag\\
&=\sum_{m\geq d(\tau)_J}\lim_{N\to\infty}\nu^m(Z(\tau\lambda^N)\cap \Lambda^{m,\infty_K}).\notag
\end{align}
Of course, we want to pull the limit through the infinite sum. So we note that
\[
Z(\tau\lambda^N)\cap \Lambda^{m,\infty_K}=\bigcup_{\mu\in s(\lambda)\Lambda^{m-d(\tau)_J,0}}Z(\tau\lambda^N\mu)\cap\Lambda^{m,\infty_K}.
\]
Then using \eqref{nu_m} we get
\begin{align*}
\nu^l(Z(\tau\lambda^N)&\cap \Lambda^{m,\infty_K})
=\sum_{\mu\in s(\lambda)\Lambda^{m-d(\tau)_J,0}}\nu^m(Z(\tau\lambda^N\mu)\cap\Lambda^{m,\infty_K})\\
&=\sum_{\mu\in s(\lambda)\Lambda^{m-d(\tau)_J,0}}e^{-r\cdot d(\tau\lambda^N\mu)}C_J\kappa_{s(\mu)}\\
&=e^{-r\cdot (m,n+N(p-n))}C_J\sum_{w\in \Lambda^0}A^{(m-d(\tau)_J,0)}(s(\lambda),w)\kappa_w\\
&=e^{-r\cdot (m,n+N(p-n))}C_J\big(A^{(m-d(\tau)_J,0)}\kappa\big)_{s(\lambda)}\\
&=e^{-r\cdot (m,n+N(p-n))}C_J\rho(\Lambda)^{(m-d(\tau)_J,0)}\kappa_{s(\lambda)}\\
&=e^{-(r-\ln\rho(\Lambda))\cdot(m,0)}\rho(\Lambda)^{-(d(\tau)_J,n+N(p-n))}C_J\kappa_{s(\lambda)}\\
&\hspace{6cm}\text{since $e^{r_i}=\rho(A_i)$ for $i\in K$}\\
&\leq e^{-(r-\ln\rho(\Lambda))\cdot (m,0)}C_J\kappa_{s(\lambda)}\quad\text{since $\rho(A_i)\geq 1$.}
\end{align*}
Since $r_j>\ln\rho(A_j)$ for $j\in J$, the series $\sum_{m\in \NN^J}e^{-(r-\ln\rho(\Lambda))\cdot (m,0)}$ converges. Thus the dominated convergence theorem and \eqref{estphi} imply that
\begin{equation}\label{estphifinal}
\psi(t_\sigma t^*_\tau)\leq\lim_{N\to\infty}\sum_{m\in\NN^J}\nu^m(Z(\tau\lambda^N)\cap \Lambda^{m,\infty_K}).
\end{equation}
Equation \eqref{eq-41} implies that the sum on the right-hand side of \eqref{estphifinal} is \[
\sum_{m\in\NN^J}\nu^m(Z(\tau\lambda^N)\cap \Lambda^{m,\infty_K})=e^{-r\cdot(d(\tau)_J,n+N(p-n))}\kappa_{s(\lambda)},
\]
and this goes to $0$ as $N\to \infty$ because $p-n>0$ and (for the $i$ such that $K=\{i\}$) $r_i=\rho(A_i)>1$ because $(\Lambda^0,\Lambda^{e_i},r,s)$ is not a cycle. Thus \eqref{estphifinal} implies that $\psi(t_\sigma t_\tau^*)=0$, and this completes the proof of \eqref{kmsformula2}.

Now the last two paragraphs in the proof of Theorem~\ref{main_thm}\eqref{6c} carry over to this situation.
\end{proof}

\begin{rmk}\label{notunique} The formula \eqref{eq-formula-phi2} defines a state $\psi$ on $\Tt C^*(\Lambda)$ for every $K$ with $|K|\geq 1$. The formula \eqref{kmsformula2} says that when $|K|=1$, this state is supported on the diagonal $D:=\clsp\{T_\sigma T_\sigma^*:\sigma \in \Lambda\}$. We claim that this is not necessarily the case if $|K|\geq 2$. To see this, suppose that $|K|\geq 2$ and that the graph $\Lambda^K:=(\Lambda^0,d^{-1}(\NN^K),r,s)$ is periodic. We write $\{t^K_\lambda\}$ for the universal Toeplitz-Cuntz-Krieger family in $\Tt C^*(\Lambda^K)$. Then \[\{T^K:=t_\mu:\mu\in \Lambda^K\}\] is a Toeplitz-Cuntz-Krieger   $\Lambda^K$-family in $\Tt C^*(\Lambda)$, and hence gives  a homomorphism $\pi_{T^K}:\Tt C^*(\Lambda^K)\to\Tt C^*(\Lambda)$ such that $\pi_{T^K}(t^K_\lambda)=t_\lambda$. We are going to use the results of \cite{aHLRS3} to compute the values of the KMS state $\psi$  on elements of the form $\pi_{T^K}(t^K_\lambda (t^K_\mu)^*)$ for $\mu,\nu\in \Lambda^K$.

So suppose $\lambda,\mu\in \Lambda^K$ (so that $d(\lambda)_J=d(\mu)_J=0$). By \eqref{eq-formula-phi2}, our state satisfies
\begin{equation}\label{getphi}
\psi\circ\pi_{T^K}(t^K_\lambda (t^K_\mu)^*)=\sum_{m\in \N^J}\int_{\Lambda^{m,\infty_K}} \big( T_\lambda T_\mu^*\xi_{m,x}\,|\,\xi_{m,x}\big)\,d\nu^m(x).
\end{equation}
For each $x\in \Lambda^{m,\infty_K}$, the vector $T_\lambda T_\mu^*\xi_{m,x}$ is either a basis element or $0$. Thus the integrands on the right-hand side of \eqref{getphi} are all non-negative. Thus we have
\[
\psi\circ\pi_{T^K}(t^K_\lambda (t^K_\mu)^*)\geq\int_{\Lambda^{0,\infty_K}} \big( T_\lambda T_\mu^*\xi_{0,x}\,|\,\xi_{0,x}\big)\,d\nu^0(x)\geq 0.
\]
Recall that $C_J:=\prod_{j\in J}(1-e^{r_j}\rho(A_j))$. Then, because $d(\lambda)_J=0$, the formula \eqref{nu_m} shows that the measure $\nu^0$ satisfies
\begin{equation}\label{ourphi}
\nu^0(Z(\lambda)\cap\Lambda^{0,\infty})=e^{-r\cdot d(\lambda)}C_J\kappa_{s(\lambda)}=\rho(\Lambda^K)^{-d(\lambda)_K}C_J\kappa_{s(\lambda)}.
\end{equation}
  
Now we compare our formula for $\psi$ with that of the state in \cite[Proposition~10.2]{aHLRS3} for the preferred dynamics on the graph $\Lambda^K$. Since $\Lambda^K$ has vertex matrices $\{A_i:i\in K\}$, $\Lambda$ and $\Lambda^K$ have the same unimodular Perron-Frobenius eigenvector. The model graph $\Omega_{|K|}$ for infinite paths in $\Lambda^K$ sits inside the model graph $\Omega_{k,(0,\infty_K)}$ for paths in $\Lambda^{0,\infty_K}$, and the map $x\mapsto x|_{\Omega_{|K|}}$ is a homeomorphism $h$ of $\Lambda^{0,\infty_K}$ onto $(\Lambda^K)^\infty$. The homeomorphism $h$ carries the set $Z(\lambda)\cap \Lambda^{0,\infty_K}$ into the cylinder set $Z(\lambda)\subset (\Lambda^K)^\infty$. Comparing \eqref{ourphi} with the formula (8.3) for $M$ in \cite{aHLRS3} shows that $\nu^0$ is the measure $C_Jh_*M$ pulled over from $(\Lambda^K)^{\infty}$. The underlying bijection of  $(\Lambda^K)^{\infty}$ onto $\Lambda^{0,\infty_K}$ gives a unitary isomorphism $V$ of $\ell^2((\Lambda^K)^\infty)$ onto the summand $H_0:=\ell^2(\Lambda^{0,\infty_K})$ of $\ell^2(\partial^K\!(\Lambda))=\bigoplus_{m\in \NN^J}\ell^2(\Lambda^{m,\infty_K})$. This isomorphism $V$ maps the usual basis $\{h_x\}$ for $\ell^2((\Lambda^K)^\infty)$ into the basis $\{\xi_{0,x}\}$, and intertwines the usual infinite-path representation (denoted $\pi_S$ in \cite[\S10]{aHLRS3}) and $\pi_{T^K}|_{H_0}$. Thus
\begin{align}\label{periodiccase}
\psi(t_\lambda t_\mu^*)=\psi\circ\pi_{T^K}(t^K_\lambda (t^K_\mu)^*)
&\geq\int_{\Lambda^{0,\infty_K}} \big( T_\lambda T_\mu^*\xi_{0,x}\,|\,\xi_{0,x}\big)\,d\nu^0(x)\\
&=C_J\int_{(\Lambda^K)^{\infty}} \big(\pi_S(t_\lambda t_\mu^*)h_x\,|\,h_x\big)\,dM(x),\notag
\end{align}
which is, modulo the nonzero scalar $C_J$, the formula for the KMS$_1$ state of $(\Tt C^*(\Lambda^K),\alpha^r)$ in \cite[Proposition~10.2]{aHLRS3}, which we will denote here by $\psi^K$.

One of the main points made in \cite[\S10]{aHLRS3} was that, when $d(\lambda)-d(\mu)$ belongs to the periodicity group $\Per(\Lambda^K)$, the state $\psi^K$ does not vanish on $t^K_\lambda (t^K_\mu)^*$. Hence our estimate \eqref{periodiccase} shows that $\psi(t_\lambda t_\mu^*)$ does not vanish either.  This settles the  claim we made above. 
\end{rmk}

\section{$2$-graphs with a single vertex}

Here we illustrate our results by applying them to a $2$-graph with a single vertex. Such graphs were first studied by Kribs and Power \cite{KP2}, and their $C^*$-algebras have been extensively studied by Davidson and Yang \cite{DY, Yang1, Yang2}. Yang in particular has made a convincing case that these $C^*$-algebras should be viewed as higher-rank anaologues of the Cuntz algebras, and share many of their properties.

We suppose that $\Lambda$ is a $2$-graph with one vertex, $N_1>0$ blue edges and $N_2>0$ red edges.  Such graphs are always coordinatewise irreducible, and their spectral radii are $\ln N_1$ and $\ln N_2$. Our conventions at \eqref{choice} say that the dynamics is given by a vector $r\in (0,\infty)^2$ such that $r_i^{-1}\ln\rho(A_i)=1$ for one $i\in \{1,2\}$ and $r_j^{-1}\ln\rho(A_j)\geq 1$ for $j\not =i$. We may as well suppose that $r_2^{-1}\ln\rho(A_2)=1$ and $r_1^{-1}\ln\rho(A_1)\geq 1$ (otherwise swap the colours).

We first consider inverse temperatures $\beta$ satisfying $\beta>1$. Then we have $\beta r_i >\ln N_i$ for both $i$, and \cite[Theorem~6.1]{aHLRS2} implies that there is a single KMS$_\beta$ state. By \cite[(6.2)]{aHLRS2}, the vector $y$ is the real number given by
\begin{align*}
y_v&=\sum_{n=0}^\infty  e^{-\beta (r_1,\ln N_2)\cdot n}|\Lambda^n|\\
&=\sum_{n=0}^\infty e^{-\beta r_1 n_1}N_2^{-\beta n_2}N_1^{n_1}N_2^{n_2}\\
&=\Big(\sum_{n_1=0}^\infty(N_1 e^{-\beta r_1})^{n_1}\Big)\Big(\sum^\infty_{n_2=0}(N_2^{1-\beta})^{n_2}\Big)\\
&=\big(1-N_1e^{-\beta r_1}\big)^{-1}\big(1-N_2^{1-\beta}\big)^{-1}.
\end{align*}
So $\Sigma_\beta=\{y_v^{-1}\}$ and $m_v=(1-e^{-\beta r_1}N_1)^{-1}(1-N_2^{1-\beta})^{-1}y_v^{-1}$, which is $1$. (And which is a nice reality check, since $m_v$ is supposed to be $\phi_\epsilon(p_v)$, which is $1$, because $p_v$ is the identity of $\Tt C^*(\Lambda)$.) The single KMS$_\beta$ state is given by
\[
\phi_\epsilon(t_\mu t_\nu^*)=\delta_{\mu,\nu}e^{-r_1\beta d(\mu)_1}N_2^{-\beta d(\mu)_2}.
\]
The measure $\nu_\beta$ on $\Lambda$ implementing this (see Remark~\ref{atomicmeas}) satisfies 
\[
\nu_\beta(\{\lambda\})=e^{-r_1\beta d(\lambda)_1}N_2^{-\beta d(\lambda)_2}\big(1-N_1e^{-\beta r_1}\big)\big(1-N_2^{1-\beta}\big).
\]

Next we suppose that $r_1^{-1}\ln N_1=r_2^{-1}\ln N_2=1$ and consider $\beta=1$. Then $K=\{1,2\}$ and $\alpha^r$ is the preferred dynamics studied in \cite{aHLRS2}. The common unimodular Perron-Frobenius eigenvector of the vertex matrices is $\kappa=1$, and hence the argument in the first paragraph of the proof of \cite[Theorem~7.2]{aHLRS2} gives a KMS$_1$ state $\phi$ of $(\Tt C^*(\Lambda),\alpha^r)$ such that
\begin{equation*}
\phi(t_\mu t_\nu^*)=\delta_{\mu,\nu}N_1^{-d(\mu)_1}N_2^{-d(\mu)_2}
\end{equation*}
(as observed in Remark~7.3 of \cite{aHLRS2}). The measure $M$ implementing this state in \cite[Proposition~10.2]{aHLRS3} satisfies 
\[
M(Z(\lambda))=N_1^{-d(\lambda)_1}N_2^{-d(\lambda)_2}
\]
(see \cite[Proposition~8.1]{aHLRS3}). We do not need rational independence of $r_1=\ln N_1$ and $r_2\ln N_2$ to get existence of the KMS$_1$ state $\phi$ (as observed in Remark~7.3 of \cite{aHLRS2}), but we do need it to deduce from  \cite[Theorem~7.2]{aHLRS2} that this is the only KMS$_1$ state of $(\Tt C^*(\Lambda),\alpha^r)$, and that it factors through a state of $C^*(\Lambda)$. However, we now know from Theorem~7.1 of \cite{aHLRS3} that there is a unique KMS$_1$ state if and only if $\Lambda$ is aperiodic, and rational dependence is only a sufficient condition for that  (see Remark~\ref{periodic?}).

Finally, suppose that $r_1^{-1}\ln N_1<r_2^{-1}\ln N_2=1$ and $\beta=1$. Since $|K|=\{2\}$ has a single element, Theorem~\ref{main_thm} gives us a KMS$_1$ state $\phi$ of $(C^*(\Lambda),\alpha^r)$ satisfying 
\[
\phi(t_\mu t_\nu^*)=\delta_{\mu,\nu}e^{-r_1d(\mu)_1}N_2^{-d(\mu)_2}.
\]
Since $\Lambda^0=\{v\}$, Theorem~\ref{main_thm}\eqref{6c} implies that $\phi$ factors through the quotient by the ideal generated by the single element $t_v-\sum_{e\in \Lambda^{e_2}}t_et_e^*$. The measure $\mu$ giving the spatial realisation of $\phi$ in Corollary~\ref{cor:measure} is supported on the semi-infinite path space $\partial^{\{2\}}(\Lambda)$, and there is given by 
\[
\mu(Z(\lambda)\cap \partial^{\{2\}}\Lambda)=e^{-r_1d(\mu)_1}N_2^{-d(\mu)_2}.
\]
If $r_1^{-1}\ln N_2$ is irrational, then Theorem~\ref{main_thm}\eqref{6d} implies that $\phi$ is the only KMS$_1$ state. If $r_1^{-1}\ln N_2$ is rational, then we have no information. It seems unlikely that the uniqueness is still connected with the periodicity of the graph $\Lambda$: periodicity requires that $\ln N_2/\ln N_1$ is rational, which seems quite unrelated to rationality of $r_1^{-1}\ln N_2$.

\appendix

\section{Relative graph algebras}\label{App:RCK}

When $E$ is a row-finite graph and $V$ is a subset of $E^0$ which contains no sources, the \emph{relative graph algebra} $C^*(E,V)$ is the quotient of $\TC^*(E)$ by the ideal generated by
\[
\Big\{t_v-\sum_{e\in vE^1}t_e t_e^*:v\in V\Big\}.
\]
These were introduced by Muhly and Tomforde \cite{MT} as a tractable family of relative Cuntz-Pimsner algebras: indeed, each $C^*(E,V)$ is isomorphic to the graph algebra of a graph obtained by adding an extra vertex $v'$ for each $v\in E^0\backslash V$ and an edge $e'$ from $v'$ to $v$ for each $e\in vE^1$ \cite[Theorem~3.7]{MT}.

The algebra in Proposition~\ref{TCK} looks like a relative graph algebra, but it doesn't at first sight look much like the relative graph algebras of Sims \cite{S1}. This is mostly because he was interested primarily in extending theory to cover the finitely aligned higher-rank graphs of \cite{RSY}, and adjusting constructions and definitions to accommodate them is a complicated business. For the graphs of interest to us, things simplify, and it is quite easy to see that our algebra is indeed a relative graph algebra in his sense (Proposition~\ref{idrelCK}). It is trickier to see that our semi-infinite path representation is one of his boundary-path representations (Proposition~\ref{aidanrep}), but the calculation might provide an instructive example for anyone interested in the satiation process of \cite{S1}.

We discuss a row-finite $k$-graph $\Lambda$ with no sources. We add the notation
\[
\MCE(\lambda,\mu)=\{\sigma\in \Lambda: d(\sigma)=d(\lambda)\vee d(\mu)\text{ and $\sigma$ has the form }\sigma=\lambda\alpha=\mu\beta\};
\]
thus $(\alpha,\beta)\in \Lambda^{\min}(\lambda,\mu)$ if and only if $(\lambda\alpha,\mu\beta)\in \MCE(\lambda,\mu)$. A subset $E$ of $\Lambda$ is \emph{finite exhaustive} if it is finite, if there is a vertex $v$ (denoted $r(E)$) such that $E\subset v\Lambda$, and if for every $\mu\in v\Lambda$, there exists $\lambda\in E$ such that $\MCE(\mu,\lambda)$ is nonempty. Notice that any finite subset of $v\Lambda$ containing a finite exhaustive set $F$ is itself finite exhaustive, and that any set of the form $\{\lambda(0, n_\lambda):\lambda\in F, n_\lambda \le d(\lambda)\}$ is finite exhaustive. For the graphs we are considering, each $v\Lambda^{e_i}$ is finite exhaustive.

Sims' constructions use collections $\mathcal{F}$ of finite exhaustive sets. For each such $\mathcal{F}$, the \emph{relative Cuntz-Krieger algebra} $C^*(\Lambda;\mathcal{F})$ is the quotient of $\TC^*(\Lambda)$ by the ideal generated by
\begin{equation}\label{defrel}
\Big\{\prod_{\lambda\in E}(q_{r(E)}-t_\lambda t_\lambda^*):E\in \F\Big\}.
\end{equation}
When we take the quotient of $\TC^*(\Lambda)$ by the relations \eqref{defrel}, we impose other relations of the same form. The main point made in \cite{S1} is that these relations are those corresponding to finite exhaustive sets in the ``satiation'' of $\F$, which we helpfully define as the smallest satiated set (of finite exhaustive sets) that contains $\F$. The concept of satiated set is described in Definition~4.1 of \cite{S1}; that the satiation has the required properties is the content of \cite[Corollary~5.6]{S1} and the uniqueness theorems in \cite[\S6]{S1}.

\begin{rmk}\label{aidanvworld}
The definition of a Toeplitz-Cuntz-Krieger family in \cite{S1} is sufficiently different from the one in \cite{RS1,aHLRS2} (and ours in \S\ref{TCK-defn}) that one might want to be reassured that these papers are all about the same algebras. The projections $T_v$ in \cite{S1} are denoted by $Q_v$ in \cite{RS1,aHLRS2}, and then the relations (T1), (T2) and (T5) in \S\ref{TCK-defn} are the same as (TCK1), (TCK2) and (TCK3) in \cite[\S3]{S1}. It may appear that (T3) and (T4) are additional relations, but this is not so. The relation (T3) is subsumed in (TCK3) by interpreting $\Lambda^{\min}(\lambda,\lambda)$ as $\{s(\lambda)\}$. Similarly, by interpreting empty sums in (TCK3) as $0$, we find that $T_\lambda^*T_\mu=\delta_{\lambda,\mu}T_{s(\mu)}$, and hence $\{T_\lambda T_\lambda^*:d(\lambda)=n\}$ is a set of mutually orthogonal projections. From (TCK2) we deduce that if $r(\lambda)=v$ then
\[
T_v(T_\lambda T_\lambda^*)=T_{v\lambda}T_\lambda^*=T_\lambda T_\lambda^*=T_\lambda T_{v\lambda^*}=(T_\lambda T_\lambda^*)T_v,
\]
and hence $T_v\geq T_\lambda T_\lambda^*$ (by \cite[Proposition~A.1]{R}, for example).
Since $\MCE(\lambda,\mu)=\MCE(\mu,\lambda)$, the relation (TCK3) also implies that for any $\lambda$ and $\mu$ in $\Lambda$ we have
\[
(T_\lambda T_\lambda^*)(T_\mu T_\mu^*)=T_\lambda\Big(\sum_{(\alpha,\beta)\in \Lambda^{\min}(\lambda,\mu)}T_\alpha T_\beta^*\Big)T_\mu=\sum_{\sigma\in \MCE(\lambda,\mu)}T_\sigma T_\sigma^*=(T_\mu T_\mu^*)(T_\lambda T_\lambda^*).
\]
Now the orthogonality of the $T_\lambda T_\lambda^*$ implies that
\[
T_v\geq\sum_{\lambda\in v\Lambda^n}T_\lambda T_\lambda^*,
\]
which is (T4).
\end{rmk}

\begin{prop}\label{idrelCK}
Let $\Lambda$ be a row-finite $k$-graph with no sources. Suppose that $J\sqcup K$ is a partition of $\{1, \dots,k \}$ and
\begin{equation}\label{keysat}
\E=\big\{v\Lambda^{e_i}:v\in \Lambda^0\text{ and }i\in K\big\}.
\end{equation}
Then $C^*(\Lambda;\E)$ is the quotient of $\TC^*(\Lambda)$ by the ideal generated by
\[
\Big\{t_v-\sum_{e\in v\Lambda^{e_i}}t_e t_e^*:v\in \Lambda^0\text{ and }i\in K\Big\}.
\]
\end{prop}

\begin{proof}
Since $t_v\geq t_e t_e^*$ for $e\in v\Lambda^{e_i}$ and $(t_e t_e^*)(t_f t_f^*)=0$ for $e\not=f$ in $v\Lambda^{e_i}$ (by (T4) or by Remark~\ref{aidanvworld} if you prefer the definitions in \cite{S1}), we have
\[
\prod_{e\in v\Lambda^{e_i}}(t_{v}-t_e t_e^*)
=t_v-\sum_{e\in v\Lambda^{e_i}}t_e t_e^*,
\]
and the result follows.
\end{proof}

In \cite[\S4]{S1}, Sims constructs a family of representations of his relative graph algebras, and this construction requires that we work with satiated collections of finite exhaustive sets. So to apply his construction to the algebra of Proposition~\ref{idrelCK} we need to identify the satiation  of the set $\E$ in \eqref{keysat}.

\begin{prop}%\label{satiation}
Let $\Lambda$ be a row-finite $k$-graph with no sources. Let $J\sqcup K$ be a partition of $\{1,\dots, k\}$ and $\E$ as in \eqref{keysat}.
Then the satiation $\overline{\E}$ of $\E$ consists of the sets which contain a finite exhaustive subset lying in $d^{-1}(\N^K)$.
\end{prop}

\begin{proof}
We denote by $\F$ the collection of finite subsets $G$ of $\Lambda$ such that $G\subset v\Lambda$ for some $v\in \Lambda^0$, and such that $G$ contains a finite exhaustive subset of $d^{-1}(\N^K)$. Then we need to prove that $\F=\overline{\E}$. First we show that $\F$ is satiated by verifying the axioms (S1--4) of \cite[Definition~4.1]{S1}.

Axiom (S1) is clear. For (S2), we take $G\in \F$, $\mu\in r(G)\Lambda\setminus G\Lambda$, and write $d(\mu)=(m,n)\in \N^J \times \N^K$. Choose a finite exhaustive subset $F$ of $G$ lying in $d^{-1}(\N^K)$; by (S1) it suffices  to show that
\[
\Ext(\mu, F):=\big\{\alpha\in r(F)\Lambda:\text{there are $\lambda\in F$, $\beta\in \Lambda$ such that $(\alpha,\beta)\in \Lambda^{\min}(\mu,\lambda)$}\big\}
\]
is in $\F$, because $\Ext(\mu, G)$ is larger. Suppose $\lambda\in F$ and $(\alpha,\beta) \in \Lambda^{\min}(\mu,\lambda)$. Then $d(\mu)\vee d(\lambda)=\big(m,\max(n,d(\lambda))\big)$. If $n\geq d(\lambda)$, then $\alpha=s(\mu)$ trivally has degree in $\N^K$; if $n\leq d(\lambda)$, then $d(\alpha)=(0,n-d(\lambda))\in \N^K$. Since $\Ext(\mu, F)$ is finite exhaustive \cite[Lemma~2.3]{S1}, we have $\Ext(\mu, F)\in \F$. Thus $\F$ has property (S2).

For (S3), it again suffices to consider $F\in \F\cap d^{-1}(\N^K)$. Then any set
\[
\big\{\lambda(0,n_\lambda):\lambda\in F,\; n_\lambda\not=0,\;n_\lambda\leq d(\lambda)\big\}
\]
of initial segments is finite exhaustive and lies in $\F$, as required. Next we check (S4), which involves choosing a subset $G'$ of $G$ and replacing each $\lambda\in G'$ with a set of the form $\lambda G_\lambda$ where $G_\lambda\subset s(\lambda)\Lambda$ is finite exhaustive. By (S1) it suffices to work with $G$ and all $G_\lambda$ contained in $d^{-1}(\N^K)$. But then the new set $(G\backslash G')\cup\big(\bigcup_{\lambda\in G'}\lambda G_\lambda\big)$ is in $d^{-1}(\N^K)$. We have now shown that $\F$ is satiated.

For $i\in K$, the sets $v\Lambda^{e_i}$ are finite exhaustive and lie in $\F$, so $\overline{\E}\subset \F$. To see that $\overline{\E}=\F$, we take a satiated set $\G$ containing $\E$, and show that $\F\subset \G$. Take $G\in \F$. Then for each $i\in K$, the set $r(G)\Lambda^{e_i}$ belongs to $\G$. Thus by (S2) and (S4), so does the set $\bigcup_{e\in r(G)\Lambda^{e_i}}e\Ext(e,G)$. For each $\mu\in \Ext(e,G)$ either $e\mu$ is in $G$ or $e\mu$ has the form $\lambda f$ for some $\lambda \in G$ with $d(\lambda)\wedge e_i=0$ and $f\in s(\lambda)\Lambda^{e_i}$. Now removing all such $f$ from $\lambda f\in r(G)\Lambda^{e_i}$ gives us $G$ back, and by (S3) gives us another element of the satiated set $\G$. Thus $G\in \G$.
\end{proof}

Now Corollary~5.6 of \cite{S1} implies that $C^*(\Lambda;\E)=C^*(\Lambda;\overline{\E})=C^*(\Lambda;\F)$. Thus $C^*(\Lambda;\E)$ has an $\F$-compatible boundary-path representation as in \cite[Lemma~4.6]{S1}.

\begin{prop}\label{aidanrep}
The set $\partial(\Lambda;\F)$ of $\F$-compatible boundary paths is the same as the set $\partial^K\!\Lambda$, and the semi-infinite path representation is the $\F$-compatible boundary-path representation of Sims in \cite{S1}.
\end{prop}

\begin{proof}
Suppose $x:\Omega_{k,m}\to \Lambda$ is in $\partial(\Lambda;\F)$. The collection $\F$ contains every $v\Lambda^{e_i}$ with $i\in K$. We can apply the definition of boundary path in \cite[Definition~4.3]{S1} with $n=m$ and $E=v\Lambda^{e_i}$ only if $m+e_i\leq m$. Hence $m_i=\infty$ for all $i\in K$, and $\partial(\Lambda;\F)\subset \partial^K\!\Lambda$.

Conversely, suppose $x\in \Lambda^{p,\infty_K}$ for some $p\in \N^J$, that $n\leq (p,\infty_K)$ and that $G\in \F$ has $r(G)=x(n)$. Then $G$ contains a finite exhaustive set $F$ lying in $d^{-1}(\N^K)$, and there exists $\lambda\in F$ with $\lambda=x(n,n+d(\lambda))$. Hence $x$ is an $\F$-compatible boundary path.

Since $\ell^2(\partial(\Lambda;\F))=\ell^2(\partial^K\!\Lambda)$, the formula in Definition~4.5 of \cite{S1} shows that the boundary-path family $S_{\F}$ on $\ell^2(\partial(\Lambda;\F))$ is the family in Proposition~\ref{TCK}.
\end{proof}

\section{A groupoid model for the Toeplitz algebra of a higher-rank graph}\label{App:nesh}

In \cite{N}, Neshveyev studies KMS states of $C^*$-algebras of \'etale groupoids with dynamics arising 
from continuous $\RR$-valued cocycles on the groupoids. Both the Toeplitz algebra and the 
Cuntz-Krieger algebra of a $k$-graph admit such groupoid models. Here we claim that the measure $\mu$ constructed in Corollary~\ref{cor:measure} is the measure on the unit space of the groupoid of the Toeplitz algebra predicted by \cite[Theorem~1.3]{N}.

\subsection*{The groupoid of the Toeplitz algebra of a $k$-graph} We start by using Yeend's work \cite{Y} to show that there is a groupoid model for the Toeplitz algebra of a $k$-graph\footnote{There is an alternative groupoid model based on an inverse-semigroup action in \cite{FMY}.}. With the exception of the two appendices, we have worked with finite $k$-graphs only. Here we consider  a row-finite $k$-graph since the extra generality does not cause any technical problems and may be of independent interest.

Let $\Lambda$ be a row-finite $k$-graph. By Theorems 3.1~and~3.2 of \cite{W}, the
collection of $Z(\lambda \setminus G)$ defined at \eqref{defn-cylindersets} form a basis for a locally compact  Hausdorff topology on the path space $W_\Lambda$. For $x\in W_\Lambda$ and  $n \le
d(x)$, we define $\sigma^n(x) \in \Lambda^{d(x)-n}$ by $\sigma^n(x)(p,q) = x(p+n, q+n)$; this gives a partially defined shift map $\sigma$ on $W_\Lambda$.
The set
\begin{equation*}%\label{eq-Toeplitz groupoid}
\Gg := \{(x,p-q,y) : x,y \in W_\Lambda, p \le d(x), q \le d(y)
    \text{ and } \sigma^p(x) = \sigma^q(y)\}
\end{equation*}
is a groupoid with range and source maps $r(x,g,y) = (x,0,x)$ and $s(x, g, y) = (y,0,y)$, partially defined multiplication  $(x,g,y)(y,h,z) = (x,
g+h, z)$ and inverses $(x,g,y)^{-1} = (y,-g,x)$.  

For $\lambda,\eta \in \Lambda$ with $s(\lambda) = s(\eta)$, we define
\begin{align*}
Z(\lambda*\eta)
    = \{(x, d(\lambda) - d(\eta), y) \in \Gg : x \in Z(\lambda), y \in Z(\eta)
                    \text{ and } \sigma^{d(\lambda)}(x) = \sigma^{d(\eta)}(y)\}.
\end{align*}
For finite $G \subset s(\lambda)\Lambda$, we set
\[\textstyle
Z(\lambda*\eta \setminus G) = Z(\lambda * \eta) \setminus
                            \big(\bigcup_{\tau \in G} Z(\lambda\tau * \eta\tau)\big).
\]
We identify the unit space $\Gg^{(0)}$ with $W_\Lambda$ via $(x,0,x)\mapsto x$. The corresponding identification of  $C_0( W_\Lambda)$ with $C_0(\Gg^0)$ sends  $\chi_{Z(\lambda)}$ to $\chi_{Z(\lambda*\lambda)}$

The following proposition follows  from Yeend's results for topological 
higher-rank graphs \cite{Y}. 

\begin{prop}\label{prop:gpd model} Let $\Lambda$ be a row-finite $k$-graph.
The sets $Z(\lambda*\eta \setminus G)$ form a basis of compact open sets for a
locally compact second-countable Hausdorff topology on $\Gg$ under which it is an \'etale topological groupoid. The set $\Lambda^\infty$ is a closed $\Gg$-invariant
subset of $\Gg^{(0)}$. There is an isomorphism $\pi : \Tt C^*(\Lambda) \to C^*(\Gg)$
such that $\pi(t_\lambda) = \chi_{Z(\lambda * s(\lambda))}$, and $\pi$ restricts to an isomorphism of $\clsp\{t_\lambda t_\lambda^*:\lambda\in\Lambda\}$ onto $C_0(\Gg^0)$.

 Let $q_\infty : C^*(\Gg)
\to C^*(\Gg|_{\Lambda^\infty})$ be the quotient map induced by restriction of functions.
Then $q_\infty \circ \pi$ factors through an isomorphism $\tilde\pi : C^*(\Lambda) \to
C^*(\Gg|_{\Lambda^\infty})$.
\end{prop}
\begin{proof}
Since $\Lambda$ is row-finite, it is finitely aligned as in \cite{RSY}.
Thus, as a topological $k$-graph with the discrete topology, it is compactly aligned by
\cite[Remark~2.4]{Y}. By Proposition~3.6 of \cite{Y},  the sets $Z(\lambda*\eta \setminus G)$ form a basis  for a second-countable Hausdorff topology on $\Gg$.  Since $\Lambda$ is compactly aligned, it follows from \cite[Proposition~3.15]{Y} that the $Z(\lambda * \eta)$ are compact, and then that the topology is locally compact. The closed subsets $Z(\lambda * \eta \setminus G)
\subset Z(\lambda * \eta)$ are therefore also compact. By \cite[Theorem~3.16]{Y}, $\Gg$
is an \'etale topological groupoid.

 To see that $\Lambda^\infty$ is closed, we show that the complement is open.   Let $x \in
W_\Lambda \setminus \Lambda^\infty$. Choose $n \in \NN^k$ such that $n_i > d(x)_i$ for
some $i$. Since $\Lambda$ is row-finite, $x(0)\Lambda^n$ is finite and $Z(x(0) \setminus x(0)\Lambda^n) \subset W_\Lambda \setminus
\Lambda^\infty$ is an open neighbourhood of $x$. Thus $\Lambda^\infty$ is closed. 

To see that $\Lambda^\infty$ is $\Gg$-invariant, let $x\in \Lambda^\infty$.  Then 
\[
r(s^{-1}(\{x\}))=\big\{y\in W_\Lambda:\text{there are $m,n\in\NN^k$ with $m\leq d(y)$ and $\sigma^m(y)=\sigma^n(x)$}\big\}.
\]
But each such $\sigma^m(y)=\sigma^n(x)$ is in $\Lambda^\infty$, and hence $y\in\Lambda^\infty$ as well. Thus $r(s^{-1}(\{x\}))\subset \Lambda^\infty$, and so $\Lambda^\infty$ is $\Gg$-invariant.

A straightforward calculation shows that $\{\chi_{Z(\lambda * s(\lambda))}:\lambda\in\Lambda\}$ is a Toeplitz-Cuntz-Krieger $\Lambda$-family. 
Thus there is a homomorphism $\pi : \Tt C^*(\Lambda) \to C^*(\Gg)$ such that
$\pi(t_\lambda) = \chi_{Z(\lambda * s(\lambda))}$. It is easy to verify that $\pi(t_\lambda t_\lambda^*) = \chi_{Z(\lambda)}$.

Each 
\[\chi_{Z(\lambda * \eta \setminus
G)} = \pi\Big(t_\lambda \big(\prod_{\tau \in G} (t_{s(\lambda)} - t_\tau t^*_\tau)\big) t^*_\eta\Big)\] belongs to the range of $\pi$, so the Stone-Weierstrass theorem implies that
$\pi$ has dense range and hence is surjective. Since $\Gg$ is \'etale, $C_0(\Gg^0)$ embeds in $C^*(\Gg)$. Thus each $\pi(t_v) = \chi_{Z(v)}$ is nonzero, and for a finite $G
\subset v\Lambda \setminus \{v\}$, we have 
\[
\pi\Big(\prod_{\lambda \in G} (t_v - t_\lambda t^*_\lambda)\Big) = \chi_{Z(v \setminus G)}  \not= 0.
\] 
Now \cite[Theorem~8.1]{RS1}, applied to the product system $E_\Lambda$ of graphs as defined in \cite[Example~3.1]{RS1},  implies that $\pi$ is injective. Thus $\pi : \Tt C^*(\Lambda) \to C^*(\Gg)$ is an isomorphism.

For $v \in \Lambda^0$ and $n \in \NN^k$, we have $q_\infty \circ \pi(t_v -
\sum_{\lambda \in v\Lambda^n} t_\lambda t^*_\lambda) = \chi_{Z(v \setminus v\Lambda^n)
\cap \Lambda^\infty} = 0$. So $q_\infty \circ \pi$ factors through  a homomorphism $\tilde\pi : C^*(\Lambda) \to C^*(\Gg|_{\Lambda^\infty})$ such that $\tilde\pi (s_\lambda)$ is the characteristic function $\chi_{Z(\lambda,s(\lambda))}$. This is precisely the homomorphism of \cite[Corollary~3.5(i)]{KP}, and hence  is an
isomorphism.
\end{proof}

\subsection*{The measure of Corollary~\ref{cor:measure} and Neshveyev's theorem}
Let $r\in (0,\infty)^k$. There is a locally constant cocycle $c : \Gg \to \RR$ given by $c(x,
n, y) = r \cdot n$. This cocycle induces a dynamics $\alpha^c: \RR \to \Aut C^*(\Gg)$
 such that $\alpha^c_t(f)(x,n,y) = e^{itc(x,n,y)} f(x,n,y)$ for $f\in C_c(\Gg)$. The isomorphism $\pi : \Tt C^*(\Lambda) \to C^*(\Gg)$  of Proposition~\ref{prop:gpd model} intertwines $\alpha^c$ and the dynamics $\alpha^r$ we have been using.

For $x\in \Gg^{(0)}$, write $\Gg_x^x$ for the stability subgroup $\{(x,n,x)\in \Gg:n\in\ZZ^k\}$ and $\Gg_x$ for the subset $\{(y,n,x)\in \Gg: y\in W_\Lambda, n\in\ZZ^k\}$ of $\Gg$. Theorem~1.3 of \cite{N} describes the KMS$_1$ states of $(C^*(\Gg),\alpha^c)$ in terms of pairs $(\mu',
\psi)$ consisting of a quasi-invariant probability measure $\mu'$ on $\Gg^{(0)}$ with Radon-Nikodym
cocycle $e^{-  c}$ and a $\mu'$-measurable field $\psi = (\psi_x)_{x \in \Gg^{(0)}}$ of 
states $\psi_x : C^*(\Gg^x_x) \to \CC$ such that for $\mu'$-almost all $x \in \Gg^0$ 
we have
\begin{equation*}%\label{eq:invariance}
\psi_x(u_g) = \psi_{r(h)}(u_{hgh^{-1}})
    \text{ for all $g \in \Gg^x_x$ and $h \in \Gg_x$.}
\end{equation*}

Now let $\phi$ be the $\KMS_1$ state of $\Tt C^*(\Lambda)$ from  Theorem~\ref{main_thm}(c) and $\mu$ the measure from Corollary~\ref{cor:measure}. By \cite[Theorem~1.3]{N}, $\phi\circ\pi^{-1}$ is implemented by a unique pair $(\mu',\psi)$ as above. Burrowing into the proof of \cite[Theorem~1.1]{N} shows that $\mu'$ is the measure on $\Gg^{(0)}$ implementing the functional $\phi\circ\pi^{-1}\vert_{C(\Gg^{(0)})}$. For $\lambda\in\Lambda$ we have
\[\mu'(Z(\lambda))=\phi\circ\pi^{-1}(\chi_{Z(\lambda)})=\phi(t_\lambda t_\lambda^*)=\mu(Z(\lambda))\]
by \eqref{eq:cor_measure2}. Thus $\mu'$ is the extension to $W_\Lambda$ of the measure $\mu$ on $\partial^K\!\Lambda$.

Henceforth we view $\mu$ as a measure on $W_\Lambda$, and then Lemma~\ref{lem:measure support} applies to this $\mu$.
The lemma says that the  quasi-invariance of the  measure $\mu$ has consequences for the dynamics $\alpha^r$ and forces it to have support in $\partial^K\!\Lambda$ where $K=\{l: r_l=\ln\rho(A_l)\}$.

\begin{lem}\label{lem:measure support} Suppose that $\Lambda$ is a finite coordinatewise irreducible  $k$-graph.
Suppose that $\mu$ is a nonzero quasi-invariant probability measure on $\Gg^{(0)}=W_\Lambda$ with
Radon-Nikodym cocycle $e^{-c}$. 
\begin{enumerate}
\item\label{lem:measure support1} Then $r \ge \ln\rho(\Lambda):=(\ln\rho(A_1),\dots,\ln\rho(A_k))$. 
\item\label{lem:measure support2} Let $l\in\{1,\dots,k\}$. If 
$\mu\big(\bigcup_{\{n:n_l = \infty\}} \Lambda^n) \not= 0$, then $r_l = \ln\rho(A_l)$. In
particular, if $\mu(\Lambda^\infty) \not= 0$ then $r = \ln\rho(\Lambda)$.
\end{enumerate} 
\end{lem}

\begin{proof}
Let $v\in\Lambda^0$ and $\lambda\in v\Lambda^{e_i}$.  Then $Z(\lambda)\subset Z(v)$. Since $Z(\lambda* s(\lambda))$ is a bisection of $\Gg$ with $r\big(Z(\lambda* s(\lambda))\big)=Z(\lambda)$ and $s\big(Z(\lambda* s(\lambda))\big)=Z(s(\lambda))$, the quasi-invariance of $\mu$ gives
\[
\mu(Z(\lambda))=e^{- d(\lambda)\cdot r}\mu(Z(s(\lambda)))=e^{- r_i}\mu(Z(s(\lambda))).
\]
Let $i\in\{1,\dots,k\}$. Then
\begin{align}\label{eq:quasi-invariance}
\mu(Z(v))
    &\ge \mu\big(\textstyle{\bigsqcup_{w\in\Lambda^0}\bigsqcup_{\lambda \in v\Lambda^{e_i}w}} Z(\lambda)\big)=\sum_{w\in\Lambda^0}\sum_{\lambda \in v\Lambda^{e_i}w} \mu(Z(\lambda))\notag\\
    & = \sum_{w\in\Lambda^0} e^{- r_i} A_i(v,w) \mu(Z(w)).
\end{align}
Set $m^+ :=(\mu(Z(v))) \in
[0,\infty)^{\Lambda^0}$. Then \eqref{eq:quasi-invariance} says that $m^+$ satisfies $e^{ r_i} m^+ \ge A_im^+$.
Since $m^+ \in [0,\infty)^{\Lambda^0}$ is nonzero and $A_i$ is irreducible, the subinvariance theorem \cite[Theorem~1.6]{Sen} implies that  $e^{ r_i} \ge \rho(A_i)$. Thus  $ r \ge \ln\rho(\Lambda)$, giving \eqref{lem:measure support1}.

For \eqref{lem:measure support2}, suppose  that $X := \bigcup_{\{n:n_l = \infty\}} \Lambda^n$ has nonzero measure.  Set 
\[
m := \big(\mu(Z(v) \cap X)\big)_{v \in \Lambda^0}.
\]
Since $\mu(X)\neq 0$, there exists $u\in\Lambda^0$ such that $m_u>0$, and  hence $m\gneqq0$. Let $v\in\Lambda^0$. Since $x\in Z(v)\cap X$ implies $d(x)_l=\infty$ we have 
\begin{align*}%\label{eq:quasi-invariance-again}
\mu(Z(v)\cap X)
    &=\mu\big(\textstyle{\bigsqcup_{w\in\Lambda^0}\bigsqcup_{\lambda \in v\Lambda^{e_l}w}} Z(\lambda)\cap X\big) = \sum_{w\in\Lambda^0} e^{- r_l} A_l(v,w) \mu(Z(w)\cap X).
\end{align*}
Thus  $e^{r_l}m = A_lm$. By definition of spectral radius,  $e^{ r_l}
\leq \rho(A_l)$. Now \eqref{lem:measure support1}  gives $e^{ r_l}
= \rho(A_l)$ and $r_l= \ln\rho(A_l)$.

Next observe that $\Lambda^\infty = \bigcap_i\big(\bigcup_{n_i =
\infty} \Lambda^n\big)$. If $ r_l>\ln\rho(A_l)$ then $\mu\big(\bigcup_{n_l = \infty} \Lambda^n) = 0$, and hence $\mu(\Lambda^\infty)=0$ as well. This gives \eqref{lem:measure support2}.
\end{proof}

\end{document}